\documentclass[a4paper,twoside]{article}
\usepackage{amsthm,amscd}

\usepackage{amsmath,amssymb,amsfonts}
\usepackage[T1]{fontenc}
\usepackage{tikz}
\usetikzlibrary{trees}
\usepackage{pdfpages}
\usepackage[numbers]{natbib}
\usepackage[utf8]{inputenc}
\usepackage[english]{babel}
\usepackage{graphicx}               
\usepackage{color}                  
\RequirePackage[colorlinks=true,citecolor=blue,linkcolor=blue]{hyperref}
\usepackage{hyperref}
\usepackage[absolute]{textpos} 
\usepackage{multicol}
\usepackage{array}

\usepackage[babel=true]{csquotes} % csquotes va utiliser la langue définie dans babel

\usepackage{bbm}
\newcommand{\ind}{\mathbbm{1}}

\newtheorem{theorem}{Theorem}[section]
\newtheorem{cor}[theorem]{Corollary}
\newtheorem{lemma}[theorem]{Lemma}
\newtheorem{prop}[theorem]{Proposition}
\newtheorem{proptes}[theorem]{Properties}

\theoremstyle{definition}

\newtheorem{rmk}[theorem]{Remark}

\def \leqp{\leqslant}
\def \geqp{\geqslant}

\newcommand{\abs}[1]{\left\lvert #1 \right\rvert}

\newcommand{\defeq}{\mathrel{\mathop:}=}

\newcommand{\beq} {\begin{eqnarray*}}
\newcommand{\eeq} {\end{eqnarray*}}
\newcommand{\trm} {\textrm}
\newcommand{\tbf} {\textbf}
\newcommand{\noi} {\noindent}

\def \R{\mathbb{R}}

\def \N{\mathbb{N}}

\def \E{\mathbb{E}}
\def \P{\mathbb{P}}
\def \Var{\hbox{{\rm Var}}}

\def \Cov{\hbox{{\rm Cov}}}

\newcommand{\1}{{1\hspace{-0.2ex}\rule{0.12ex}{1.61ex}\hspace{0.5ex}}}

\title{Sensitivity analysis based on Cram\'er von Mises distance} % \thanks is optional. Insert line breaks with \\

\author{Fabrice Gamboa\footnote{Institut de Math\'ematiques de Toulouse, 118 Route de Narbonne 31062 Toulouse Cedex 9. France.
{\tt firstname.lastname@math.univ-toulouse.fr}} \and Thierry~Klein$^{*}$\footnote{ENAC - Ecole Nationale de l'Aviation Civile , Universit\'e de Toulouse, France} \and
Agn\`es~Lagnoux$^{*}$}

\begin{document}

\maketitle
\begin{abstract}
In this paper, we first study a sensitivity index that is based on higher moments and generalizes the so-called Sobol one. Further, following an idea of Borgonovo (see \cite{borgonovo2007}),  we define and study a new sensitivity index based on the Cram\'er von Mises distance. This index appears to be more general than the  Sobol one as it takes into account the whole distribution of the random variable and not only the variance. Furthermore, we study the statistical properties of its Pick and Freeze estimator.   
\end{abstract}

\section{Introduction}
\label{sinto}
A very classical problem in the study of computer code experiments (see  \cite{sant:will:notz:2003}) is the evaluation of the relative influence of the input variables on some numerical result obtained by a computer code. In this context, a sensitivity analysis is performed. Such a topic has been widely studied in the last decades and is still challenging nowadays (see for example \cite{sobol1993,saltelli-sensitivity,rocquigny2008uncertainty} and references therein).
More precisely,
the result of the numerical code $Y$ is seen as a function of the vector of the distributed input $(X_i)_{i=1,\cdots,d}$ ($d\in\N^*$). Statistically speaking, we are dealing here with the following unnoisy non parametric model
\begin{equation*}
Y=f(X_{1},\ldots, X_{d}),
%\label{momodel}
\end{equation*}
where $f$ is a regular unknown numerical function on the state space  $E_1\times E_2\times \ldots \times E_d$ on which the distributed variables $(X_{1},\ldots, X_{d})$ are living. Generally, the random inputs are assumed to be independent and a sensitivity analysis is performed using the so-called Hoeffding decomposition (see \cite{van2000asymptotic,anton84}). In this functional decomposition, $f$ is expanded as an $L^2$-sum of uncorrelated functions involving only a part of the random inputs. This leads, for any subset $v$ of $I_d:=\{1,\ldots,d\}$, to an index called the Sobol index (\cite{sobol1993}) that measures the amount of {\it randomness} (more precisely, the part of the variance) of $Y$   due to the subset of input variables $(X_i)_{i\in v}$. Since nothing has been assumed on the nature of the inputs, one can consider the vector $(X_i)_{i\in v}$ as a single input. Without loss of generality, we then consider the case where $v$ reduces to a singleton.
More precisely, the numerator $H_{v}^2$ of the Sobol index related to  the input $X_{v}$ is

\begin{equation*}
H_{v}^2\defeq \Var\left(\E\left[Y|X_{v}\right]\right)
\end{equation*}

while the denominator of the index is nothing more than the variance of $Y$. Notice that we also have:
\begin{equation}
H_{v}^2=\E\left[\left(\E[Y|X_v]-\E[Y]\right)^{2}\right]=\Var(Y)-\E\left[\left(\E[Y]-\E\left[Y|X_{v}\right]\right)^{2}\right]
\label{trucmoche}
\end{equation} 
In order to estimate $H_{v}^2$, Sobol in \cite{sobol1993} proposed to rewrite the variance of the conditional expectation as a covariance (see equation \eqref{sobol_cov}).
Further,  a well tailored design of experiment called the Pick and Freeze scheme is considered \cite{janon2012asymptotic}. More precisely, let $X^v$ be the random  vector such
that $X^v_v=X_v$ and $X^v_i=X'_i$ if $i\neq v$ where $X'_i$ is an independent copy of $X_i$. Then, setting 
\begin{align}\label{def:Yv}
Y^{v}:=f(X^v),
\end{align}
 an obvious computation leads to the following relationship  (see, e.g., \cite{janon2012asymptotic})
\begin{equation}\label{sobol_cov}
\Var (\mathbb{E}[Y|X_v])=\Cov\left(Y,Y^{v}\right).
\end{equation}
The last equality leads to a natural Monte-Carlo estimator, the so-called Pick and Freeze estimator,
\begin{align*}%\label{esteffgen}
T^{v}_{N, \mathrm{Cl}}&=  \frac{1}{N} \sum_{j=1}^N  Y_j    Y_j^{{v}} - \left(\frac{1}{2N} \sum_{j=1}^N (Y_j+Y_j^{v})\right)^2
\end{align*}
where for $j=1,\cdots, N$, $Y_j$ (resp. $Y_j^{v}$) are independent copies of 
$Y$ (resp. $Y^{v}$).  The sharp statistical properties and some functional extensions of the Pick and Freeze method are considered in \cite{janon2012asymptotic,pickfreeze,radouche}. 
Notice that the Sobol indices and their Monte-Carlo estimation are  order two methods since they derive from the $L^2$-Hoeffding functional decomposition. This is their main drawback. As an illustration consider the following example.
Let $X_{1}$ and $X_{2}$ be two independent standardized  random variables having the same  third and fourth moments with $\E\left[ X_{1}^{5}\right]\neq \E\left[ X_{2}^{5}\right]$. Let us consider the following model
\[
 Y=X_{1}+X_{2}+X_{1}^{2}X_{2}^{2}.
\]
One gets
\[
\Var\left(\E\left[Y|X_{1} \right]\right)=\Var(X_{1}+X_{1}^{2})=\Var(X_{2}+X_{2}^{2})=\Var\left(\E\left[Y|X_{2} \right]\right).
\]
$Y$ is an exchangeable function of the inputs but $X_{1}$ and $X_{2}$ do not share the same distribution. So that,   $X_{1}$ and $X_{2}$ should not have the same importance. That shows the need to introduce a sensitivity index that takes into account all the distribution and not only the second order behavior.  
As pointed out before,  Sobol  indices are based on an $L^{2}$ decomposition. As a matter of fact, they are  well adapted to measure the contribution of an input on the deviation around the mean of $Y$. Nevertheless, it seems  very intuitive that the sensitivity of an   extreme quantile of $Y$ could depend on sets of variables that cannot be captured using only the variances. Thus the same index should not be used for any task and we need to define more general indices.

There are several ways to generalize the Sobol indices. For example, one can define new indices through  contrast functions based on the quantity of interest (see \cite{FKR13}). Unfortunately the Monte-Carlo estimators of these indices are computationally  very expensive. In \cite{DaVeiga13}, Da Veiga presents a way to define moment independent measures through dissimilarity distances. 
These measures define a unified framework that encompasses some known sensitivity indices. They are efficiently estimated in low dimensions but as claimed by the author ``it is well known that density estimation suffers from the curse of dimensionality''.
%Unfortunately, the estimation of such indices relies on the density ratio estimation that can be computationally expensive.
Now, as pointed out in \cite{borgonovo2007,grobobo,petibobo,ODC13,Owen12}, there are situations where higher order methods give a sharper analysis on the relative influence of the input and allow finer screening procedures. Borgonovo {\it et al.} propose and study an index based on the total variation distance (see \cite{borgonovo2007,grobobo,petibobo}); while Owen {\it et al.} suggest to use procedures based on higher moments (see \cite{ODC13,Owen12}).  

Our paper follows these tracks. 
We will first revisit the work of Owen {\it et al.} by studying the asymptotic properties of the multiple Pick and  Freeze estimation.
Further,  we propose a new natural  index based on the Cram\'er von Mises distance between the distribution of the output $Y$ and its conditional law when an input is fixed. 
We will show that this approach leads to natural self-normalized indices.  Indeed, as for Sobol indices, the sum of all first order indices is uniformly bounded. Notice that these indices take into account the whole output distribution instead of only the order two moments and contrary to most of the other known indices, they are naturally defined for multivariate outputs. As a consequence, they are well-tailored to perform a sensitivity analysis for any vectorial output. 
Furthermore, we  show that surprisingly a Pick and Freeze scheme is also available to estimate this new index. This scheme is not really expensive and easy to implement. The sample size required for the estimation is of the same order as the size needed for the classical Sobol index estimation allowing its use in concrete situations. As a consequence, considering a sample with the appropriate size, one can provide simultaneously the Cramér von Mises indices and the Sobol indices. Other advantage of the Cramér von Mises index with respect to the general ones presented in \cite{DaVeiga13} is that the theoretical statistical properties of its estimation can be derived. Indeed, we prove a Central Limit Theorem for the estimator that allows one to exhibit confidence intervals.  

The paper is organized as follows. In the next section, we will study the statistical properties of the multiple Pick and Freeze method proposed earlier by Owen et al (\cite{ODC13,Owen12}). Section \ref{sec:Cramér} is devoted to the new index built on the Cram\'er von Mises distance. In the last section, we give some numerical simulations that illustrate the interest of the new index. Moreover, we revisit a real data example introduced in \cite{BD92} and studied in \cite{FH04,BHP14}.

\section{Higher-moment indices}\label{sec:pickfreeze}

In the sequel, for any integer $k$, the notation $I_k$ stands for the set $\{1,\ldots,k\}$.
%Using the classical Hoeffding decomposition, for a singleton $v\in I_{d}$, the numerator of the classical Sobol index with respect to $v$ is given by
%\begin{align*}%\label{trucmoche2}
%H_{v}^{2}&=\E\left[\left(\E[Y|X_v]-\E[Y]\right)^{2}\right].
%\end{align*}
%
%
Following \cite{ODC13,Owen12}, we  generalize the numerator of the Sobol index defined in \eqref{trucmoche} by considering higher order moments: 
for any integer $q\geqp 2,$ and singleton $v\in I_{d}$,  we set
 \begin{align*}%\label{trucmochep}
H^{q}_{v}:=\E\left[\left(\E[Y|X_v]-\E[Y]\right)^{q}\right].
\end{align*}

\paragraph{Properties}

Obviously, $H^{q}_{v}$ is non negative only  for even $q$ and 
\[
\abs{H^{q}_{v}}\leqp \E\left[\abs{Y-\E[Y]}^q\right].
\]
Further, $H^{q}_{v}$ is invariant by any 
translation of the output.

\paragraph{Estimation procedure}

In view of the estimation of  $H^{q}_{v}$, we first notice that 
\begin{align*}
H^{q}_{v}
%&=\E\left[\left(\E[Y]-\E[Y|X_i, i\in v]\right)^{p}\right]=\E\left[\prod_{i=1}^{p}\left(Y^{v,i}-\E[Y]\right)\right]\\
&  =\E\left[\prod_{i=1}^{q}\left(Y^{v,i}-\E[Y]\right)\right] =\sum_{l=0}^{q }\binom{q}{l}(-1)^{q-l} \E\left[Y \right]^{q-l}\E\left[\prod_{i=1}^{l}Y^{v,i}\right]
%
%%v\subset I_{d},
\end{align*}
with the usual convention $\prod_{i=1}^0 Y^{v,i}=1$ and $ \binom{q}{l}=q!/l!(q-l)!$. Here,  $Y^{v,1}=Y$ and for $i=2,\ldots,q$,  $Y^{v,i}$ is constructed independently as  $Y^{v}$ defined in Equation \eqref{def:Yv}.

Second, we use a Monte-Carlo scheme and consider the following Pick and Freeze design constituted by 
a $N$-sample $\begin{pmatrix} Y^{v,i}_{j}\end{pmatrix}_{(i,j)\in I_{q}\times I_{N}}$ of $\left(Y^{v,1},\ldots, Y^{v,q} \right)$.
The Monte-Carlo estimator is then
\begin{align*}%\label{trucserieux}
H^{v}_{q,N}=\sum_{l=0}^{q}\binom{q}{l}(-1)^{q-l} \left(\overline{P}^{v}_{1}\right)^{q-l} \overline{P}^{v}_{l}
\end{align*}
where for any $N\in \N^{*}$, $j\in I_N$ and $l\in I_{q}$, we have defined
%\begin{align*}
%\overline{Y^{v,i}}&=\frac{1}{N}\sum_{j=1}^{N} Y^{v,i}_{j},\quad 
%Z^{ v}=\frac1{q}\left(\sum_{i=1}^q  \overline{Y^{v,i}}\right)\ \mathrm{and\ } P^{v}_{l}=\frac{1}{N}\sum_{j=1}^{N}\left(\prod_{i=1}^{l} Y^{v,i}_{j}\right).
%\end{align*}
%
\begin{align*}
P^{v}_{l,j}=\binom{q}{l}^{-1}\sum_{k_1<\ldots < k_l\in I_q}\left(\prod_{i=1}^{l} Y^{v,k_i}_{j}\right) \quad \mbox{ and } \quad \overline{P}^{v}_{l}=\frac{1}{N}\sum_{j=1}^{N} P^{v}_{l,j}.
\end{align*}

Notice that we generalize the estimation procedure of \cite{pickfreeze} and use all the available information by considering the means over the set of indices $k_1,\ldots, k_l\in I_d,\; k_n\neq k_m$.

\paragraph{Asymptotic properties of $H^{v}_{q,N}$}

\begin{theorem}\label{th:TCL}
$H^{ v}_{q,N}$ is strongly consistent and asymptotically Gaussian:
\begin{align*}%\label{eq:TCL}
\sqrt{N}\left(H^{v}_{q,N} - H^{v}_{q}\right)
\overset{\mathcal{L}}{\underset{N\to\infty}{\rightarrow}}\mathcal{N}\left(0,\sigma^2 \right)
\end{align*}
where 
\[
\sigma^2=q \left[\Var(Y)+(q-1)\Cov(Y,Y^{v,2})\right]\left(\sum_{l=1}^q a_lb_l\right)^2,
\]
\[
a_l=\frac{l}{q}\E[Y]^{l-1},\qquad l=1,\ldots, q
\]
\[
b_1=(-1)^{q-1}q(q-1) \E[Y]^{q-1}+\sum_{l=2}^{q-1}\binom{q}{l}(-1)^{q-l} (q-l) \E[Y]^{q-l-1} \E\left[\prod_{i=1}^l Y^{v,i}\right]
\]
and
\[
b_l=\binom{q}{l}(-1)^{q-l}\E[Y]^{q-l},\qquad l=1,\ldots, q.
\]
\end{theorem}

\paragraph{Interpretation and comments}

The collection of all indices $\left(H^{q}_{v}\right)_q$ is much more informative than the classical Sobol index with respect to $v$. Nevertheless, it has several drawbacks. First, these indices are moment-based and it is well known that they are not stable when the moment order increases. Second, they may be negative when $q$ is odd. To overcome this fact, 
one may introduce $\E\left[\abs{\E[Y|X_v]-\E[Y]}^{q}\right]$ but the Pick and Freeze estimation procedure is then lost.
Third, the Pick and Freeze estimation procedure is computationally expensive and may be unstable: it requires a $q\times N$-sample of the output $Y$. In order to have a good idea of the influence of an input on the law of the 
output, we need to estimate the first $K-1$ indices $H^{q}_{v}$: $H^{2}_{v}$, \ldots, $H^{K}_{v}$. Hence, we need to run the code $K\times N$ times. 

In a nutshell, these indices are not attractive in a practical point of view. In the next section, we then introduce a new sensitivity index that is based on the conditional distribution of the output and requires only $3\times N$ runs. Concretely, it compares the output distribution to the conditional one whereas the $q$ higher-order moment indices only compare the $q$-th output moment to the conditional one.

\section{The Cram\'er von Mises index}
\label{sec:Cramér}

In this section, the code will be denoted by $Z=f(X_{1},\ldots, X_{d})\in\R^k$. It is worth noticing that here we can consider multivariate outputs unlike in Section \ref{sec:pickfreeze} and \cite{BHP14}, e.g., Let $F$ be the distribution function of $Z$:
\[
F(t)=\P\left(Z\leqp t\right)=\E\left[\1_{\{Z\leqp t\}}\right],\; \textrm{for $t=(t_1,\ldots, t_k)\in\R^k$}
\]
and $F^{v}$ be the conditional distribution function of $Z$ conditionally on $X_{v}$:
\[
F^{v}(t)=\P\left(Z\leqp t|X_{v}\right)=\E\left[\1_{\{Z\leqp t\}}|X_{v}\right],\; \textrm{for $t=(t_1,\ldots, t_k)\in\R^k.$}
\]
Notice that $\{Z\leqp t\}$ means that $\{Z_1\leqp t_1, \ldots, Z_k\leqp t_k\}$.
Obviously,  $\E\left[F^{v}(t)\right]=F(t)$.  Now, we define  $Y(t)=\1_{\{Z\leqp t\}}$ and take $p=2$. Since for any fixed $t\in \R^{k}$, $Y(t)$ is a real-valued random variable, we apply the framework presented in Section \ref{sec:pickfreeze}. More precisely, for any $v\in I_p$ let $\sim v$ be $I_p\setminus \{v\}$ and  we first perform the Hoeffding decomposition of $Y(t)$:
\begin{align}\label{eq:decomp_Y}
Y(t)&=\1_{\{Z\leqp t\}}=\E[Y(t)]+\left(\E[Y(t)|X_v]-\E[Y(t)]\right)+\left(\E[Y(t)|X_{\sim v}]-\E[Y(t)]\right)+R(t,v)
\end{align}
where 
\[
R(t,v)=Y(t)-\E[Y(t)]-\left(\E[Y(t)|X_v]-\E[Y(t)]\right)-\left(\E[Y(t)|X_{\sim v}]-\E[Y(t)]\right).
\]
As done usually, we compute the variance of both sides of \eqref{eq:decomp_Y} that leads to
\begin{align}
\Var(Y(t))&=F(t)(1-F(t))\nonumber\\
&=\Var\left(\E[Y(t)|X_v]-\E[Y(t)]\right)+\Var\left(\E[Y(t)|X_{\sim v}]-\E[Y(t)]\right)+\Var(R(t,v))\nonumber\\
&= \Var\left(F^{v}(t)\right)+ \Var\left(F^{\sim v}(t)\right)+\Var(R(t,v))\nonumber\\
&= \E\left[\left(F^{v}(t)-F(t)\right)^{2}\right]+\E\left[\left(F^{\sim v}(t)-F(t)\right)^{2}\right]+\Var(R(t,v)) \label{eq:decomp_var}
\end{align}
by the decorrelation of the different terms involved in the Hoeffding decomposition.

\begin{rmk}
A straightforward application of the results of Section \ref{sec:pickfreeze} provides for any fixed $t\in \R^{k}$ a consistent and asymptotically normal procedure for the estimation of 
\[
\E\left[\left(F^{v}(t)-F(t)\right)^{2}\right]=\Var\left(F^{v}(t)\right) \quad \text{and} \quad \E\left[\left(F^{\sim v}(t)-F(t)\right)^{2}\right]=\Var\left(F^{\sim v}(t)\right).
\]
\end{rmk}

Now we integrate the terms in \eqref{eq:decomp_var} in $t\in \R^k$ with respect to the distribution of $Z$:
\begin{align}
&\int_{\R^k} F(t)(1-F(t))dF(t)\nonumber\\
= &\int_{\R^k} \E\left[\left(F^{v}(t)-F(t)\right)^{2}\right]dF(t)+\int_{\R^k} \E\left[\left(F^{\sim v}(t)-F(t)\right)^{2}\right]dF(t)+\int_{\R^k} \Var(R(t,v)) dF(t) \label{eq:decomp_int}
\end{align}
This integration has to be understood in the Riemmann-Stieltjes sense (see, e.g., \cite{horst86}). Notice that the first term in the right hand side of \eqref{eq:decomp_int} represents a Cram\'er Von Mises type distance of order $2$ between the distribution $\mathcal{L}\left(Z\right)$  of $Z$ and the distribution $\mathcal{L}\left(Z|X_{v}\right)$  of $Z$ given $X_{v}$.\\

Following the classical way of defining Sobol indices, we normalize the previous equation by 
\[
\int_{\R^k} F(t)(1-F(t))dF(t)
\]
leading to
\begin{align}
1
&= \frac{\int_{\R^k} \E\left[\left(F^{v}(t)-F(t)\right)^{2}\right]dF(t)}{\int_{\R^k} F(t)(1-F(t))dF(t)}+\frac{\int_{\R^k} \E\left[\left(F^{\sim v}(t)-F(t)\right)^{2}\right]dF(t)}{\int_{\R^k} F(t)(1-F(t))dF(t)}+\frac{\int_{\R^k} \Var(R(t,v)) dF(t)}{\int_{\R^k} F(t)(1-F(t))dF(t)} \label{eq:decomp_norm}
\end{align}

Then we define the Cram\'er Von Mises indices with respect to $v$ and $\sim v$ by
\begin{align*}%\label{CVM}
S_{2,CVM}^{v}:=\frac{\int_{\R^{k}}\E\left[\left(F(t)-F^{v}(t)\right)^{2}\right]dF(t)}{\int_{\R^k} F(t)(1-F(t))dF(t)} \quad \text{and} \quad 
S_{2,CVM}^{\sim v}:= \frac{\int_{\R^{k}}\E\left[\left(F(t)-F^{\sim v}(t)\right)^{2}\right]dF(t)}{\int_{\R^k} F(t)(1-F(t))dF(t)}.
\end{align*}

\begin{proptes}\label{prop:prop} These new indices are naturally adapted to multivariate outputs and they share the same  properties as the classical Sobol index. Namely,
\begin{enumerate}
\item as seen in \eqref{eq:decomp_norm}, the different contributions sum to 1. 
\item they are invariant by translation, by any isometry and by any non degenerated scaling of the components of $Y$. 
\end{enumerate}
\end{proptes}

\begin{rmk}
\begin{enumerate}
\item We could have defined the following indices instead
\begin{align*}%\label{CVM}
\int_{\R^{k}}\frac{\E\left[\left(F(t)-F^{v}(t)\right)^{2}\right]}{F(t)(1-F(t))} dF(t) \quad \text{and} \quad 
\int_{\R^{k}}\frac{\E\left[\left(F(t)-F^{\sim v}(t)\right)^{2}\right]}{F(t)(1-F(t))}dF(t).
\end{align*}
normalizing by $F(t)(1-F(t))$ (like in the Anderson-Darling statistic) before the integration phase. Nevertheless, the previous integrals might be not defined. Moreover, even if the integrals are well defined, one may encounter numerical explosion during the estimation procedure that might be produced for small and large values of $t$ since the normalizing factor then cancels. 
\item In this paper, we only consider first-order sensitivity indices as well for the classical Sobol indices and for the Cramér von Mises indices.
Anyway, as well as for the Sobol indices, one may define 
higher-order and total Cramér von Mises indices. The construction of the former is straightforward taking $v$ no longer a singleton. For example, if one is interested in the second-order Cramer von Mises index with respect to the first and second inputs, it suffices to take $v=\{1,2\}$. Concerning the latter, the total Cramér von Mises index $S_{2,CVM}^{Tot,v}$ with repect to $v$ is defined by
\[
S_{2,CVM}^{Tot,v}:=1-S_{2,CVM}^{\sim v}=1-\frac{\int_{\R^{k}}\E\left[\left(F(t)-F^{\sim v}(t)\right)^{2}\right]dF(t)}{\int_{\R^k} F(t)(1-F(t))dF(t)} .
\]
\item To use the Hoeffding decomposition, the inputs are required to be independent. 
Anyway, one can compute the Cramér von Mises index when the inputs are dependent. Nevertheless, there are then difficult to interpret.
\end{enumerate}
\end{rmk}

\subsection{General comments on the Cramér von Mises indices}

\paragraph{Cramér von Mises indices versus Sobol indices}

\noi\\
Cramér von Mises and Sobol indices are both based on the Hoeffding decomposition and sum to 1. Nevertheless, the former 
are based on the whole distribution of the output, in contrast with the latter that are only based on the order-two moments.
Notice that two variables that have a different influence on the output may have the same Sobol indices (just as two random variables with different 
distribution can have the same variance). This point represents one limitation of Sobol indices and does not occur with the Cramér von Mises indices as one can see in Section \ref{ssec:toy}.   \\

In addition, remark that a null value for a Sobol index does not imply that the input is unimportant whereas a null value for a Cramér von Mises index means that the input is unimportant. Moreover, by definition, a large Cramér von Mises index means that the input variable under concern has a great influence on the output in regions 
where the output has a large distribution mass. That is why we advice the practitioner to use them in a general context. Nevertheless, when one is interested in the mean output behavior, the Sobol indices are more adapted. Indeed, as noted in \cite{FKR13}, the Sobol indices minimize the contrast associated to the mean. In the same spirit, if one is interested in specific feature of the output (for example an $\alpha$-quantile), one should use the index based on the associated contrast. See \cite{FKR13} for more details on the notion of contrast and the results therein. \\

In contrast, 
the indices based on the whole distribution partially get rid of such limitations and pathological patterns. However, one can build an example based, e.g., on two input variables that leads to the same indices $S_{2,CVM}^1$  and $S_{2,CVM}^2$ once the integration with respect to $t$ has been done. \\

%Roughly speaking, quantifying the variability of an output $Z$ through its variance, the classical Sobol index, usually denoted $S^v$, represents the part (proportion) of the total variance that is explained by the random input $X_v$ (without the interaction with the other random inputs) for a given $v\in I_d$. 
%Summing the index $S^v$ and the higher-order indices (see \cite{homma1996importance}), one can build a picture of the importance of each variable in determining the output variance.
%Now, similarly, if we quantify the variability of $Z$ through the Cramér von Mises distance, then 
%$S_{2,CVM}^{v}$ represents the part of the variability of the output $Z$ due to $X_v$. \\

\paragraph{Cramér von Mises indices versus moment independent indices}

\noi\\
There already exists several moment-independent indices: some of them have been introduced by Borgonovo et {\it al.}\ (density-based indices \cite{grobobo}, cumulative distribution function based indices \cite{BI16}). See also \cite{BB13} for other indices and references therein. More recently, Da Veiga \cite{DaVeiga13} shows that those indices are special cases of a class of sensitivity indices based on the Csiz\'ar $f$-divergence. A lot of classical ``distances'' between probability measures as, e.g., the Kullback-Leibler divergence, the Hellinger distance and the total variation distance belong to this family of divergences. Other dissimilarity measures exist to compare probability distributions: in particular, integral probability metrics \cite{Muller97}.\\

In comparison with the indices defined in Equation (17) in \cite{BI15}, we can notice that the integration is done with respect to the distribution of the output in the former indices while the integration is done with respect to the Lebesgue measure in the latter indices. Our method represents at least two advantages: (i) the index always exist whatever the output distribution (ii) such an integration weights the support of the output distribution.\\

Since the space of the probability measures on $\R^k$ is of infinite dimension, the different distances on this space are not equivalent; hence they are very difficult to compare. Each index is constructed on a specific distance and has its own interest.     
Despite the fact that the Cramér von Mises indices have no clear dual formulation,  they present the following remarkable advantages. As we will see in the next sections, one can easily estimate them with a low simulation cost that does not depend on the dimension of the output. The sample required for their estimation also provide Sobol indicies estimation. In addition, these estimators are asymptotically normal and converge at the rate $\sqrt N$ which allows the practitioner to build confidence intervals.

\noi\\
The rest of the section is dedicated to the estimation of $S_{2,CVM}^{v}$ (and $S_{2,CVM}^{\sim v}$). One has to estimate both the numerator and the denominator of the indices. Nevertheless, when the output $Z$ has independent coordinates that are absolutely continuous with respect to the Lebesgue measure, we have
\[
\int_{\R^k} F(t)(1-F(t))dF(t)=\E[F(Z)(1-F(Z)]=\frac{1}{2^k}-\frac{1}{3^k}.
\]
Thus the normalizing factor reduces to $\frac{1}{2^k}-\frac{1}{3^k}$. 
As a consequence, we propose two versions of Central Limit Theorems: the first one deals with the numerator's estimator and can be applied when  the output $Z$ has independent coordinates that are absolutely continuous with respect to the Lebesgue measure whereas the second one concerns the general estimator and may apply in any other cases.

\subsection{Numerator estimation and its asymptotic properties}\label{ssec:num_est}

We denote  the numerator of $S_{2,CVM}^{v}$ by $N_{2,CVM}^{v}$. Notice that it can be rewritten as
\begin{align*}%\label{CVM_esp}
N_{2,CVM}^{v}=\E_{\tilde Z}\left[\E_{X_v}\left[\left(F(\tilde{Z})-F^{v}(\tilde{Z})\right)^{2}\right]\right]
\end{align*}
where $\tilde{Z}$ is an independent copy of $Z$. 

Then we proceed to a double Monte-Carlo scheme for the estimation of $N_{2,CVM}^{v}$ and consider the following design of experiment consisting in:
\begin{enumerate}
\item The classical Pick and Freeze sample, that is two $N$-samples of $Z$: $(Z_j^{v,1},Z_j^{v,2})$, $1\leqp j\leqp N$;
\item A third $N$-sample of $Z$ independent of $(Z_j^{v,1},Z_j^{v,2})_{1\leqp j\leqp N}$: $W_k$, $1\leqp k\leqp N$.
\end{enumerate}

The empirical estimator of $N_{2,CVM}^{v}$ is then given by
\begin{align}\label{def:CvM_est}
\widehat N_{2,CVM}^{v}=\frac 1N \sum_{k=1}^N \left\{\frac 1N \sum_{j=1}^N  \1_{\{Z_j^{v,1}\leqp W_k\}} \1_{\{Z_j^{v,2}\leqp W_k\}}-\left[\frac {1}{2N} \sum_{j=1}^N  \left(\1_{\{Z_j^{v,1}\leqp W_k\}}+ \1_{\{Z_j^{v,2}\leqp W_k\}}\right)\right]^2\right\}.
\end{align}

Now we established the consistency of $\widehat N_{2,CVM}^{v}$ that follows directly from an auxiliary lemma (see Section \ref{sec:proofs}).

\begin{cor}\label{cor:cons_N2_est}
$\widehat N_{2,CVM}^{v}$ is strongly consistent as $N$ goes to infinity.
\end{cor}

Now we turn to the asymptotic normality of $\widehat N_{2,CVM}^{v}$. We follow van der Vaart \cite{van2000asymptotic} to establish the following proposition (more precisely Theorems 20.8 and 20.9, Lemma 20.10 and Example 20.11).

\begin{theorem}\label{th:as_norm_N2_est}
The sequence of estimators $\widehat N_{2,CVM}^{v}$ is asymptotically Gaussian in estimating $N_{2,CVM}^{v}$. That is, 
$\sqrt{N}\left(\widehat N_{2,CVM}^{v}- N_{2,CVM}^{v}\right)$ converges in distribution towards the centered Gaussian law with a limiting variance $\xi^2$ whose explicit expression can be found in the proof.
\end{theorem}

\begin{rmk}
Thanks to Theorem \ref{th:as_norm_N2_est}, we are now able to provide asymptotic confidence intervals for the estimation of  $N_{2,CVM}^{v}$. They are of the form
$(\widehat N_{2,CVM}^{v}\pm  z_{\alpha}\xi/\sqrt{N}),$ where $z_{\alpha}$ is the $1-\alpha/2$ quantile of a standard normal distribution.
Unfortunately, the variance $\xi^2$ is unknown but thanks to its explicit form it is easy to replace it by a consistent estimator  $\widehat{\xi}$ and use Slutsky's Lemma to have an asymptotic confidence interval.
\end{rmk}

\subsection{Estimation of the general index and its asymptotic properties}

In order to estimate the general index $S_{2,CVM}^{v}$, we first estimate its numerator as in Subsection \ref{ssec:num_est} and then its denominator
that we denote $D_{2,CVM}^{v}$. Notice that it can be rewritten as
\begin{align*}%\label{CVM_esp_den}
D_{2,CVM}^{v}=\E\left[F(Z)(1-F(Z))\right]
\end{align*}
and estimated using the design of experiment already introduced for the estimation of the numerator by
\begin{align}\label{def:CvM_est_den}
\widehat D_{2,CVM}^{v}=\frac 1N \sum_{k=1}^N \left\{\frac{1}{2N} \sum_{j=1}^N  \left(\1_{\{Z_j^{v,1}\leqp W_k\}}+\1_{\{Z_j^{v,2}\leqp W_k\}}\right)-\left[\frac {1}{2N} \sum_{j=1}^N  \left(\1_{\{Z_j^{v,1}\leqp W_k\}}+ \1_{\{Z_j^{v,2}\leqp W_k\}}\right)\right]^2\right\}.
\end{align}

Proceeding as in Subsection \ref{ssec:num_est}, we have

\begin{cor}\label{cor:cons_S2_est}
$\widehat S_{2,CVM}^{v}$ is strongly consistent as $N$ goes to infinity.
\end{cor}

The following Central Limit Theorem comes from the functional Delta method.

\begin{theorem}\label{th:as_norm_S2_est}
The sequence of estimators $\widehat S_{2,CVM}^{v}$ is asymptotically Gaussian in estimating $S_{2,CVM}^{v}$. That is, 
$\sqrt{N}\left(\widehat S_{2,CVM}^{v}- S_{2,CVM}^{v}\right)$ converges in distribution towards the centered Gaussian law with a limiting variance that can be computed.
\end{theorem}

\subsection{Practical advices}

\noi\\
In a general setting, for all the nice properties of the Cramér von Mises indices and their efficient estimation easy to implement, we recommend to use the Cramér von Mises indices. As a consequence, considering a sample with the appropriate size, one can estimate once at a time the Cramér von Mises indices and the Sobol indices. More precisely, if one wants to estimate $p$ Sobol indices a sample size of $(p+1) N$ is required. With only $N$ more output evaluations, we get both the $p$ Sobol indices  and the Cramér von Mises ones. Furthermore, the theoretical theorems  provides confidence intervals that controlled the accuracy of the estimations.
Anyway, when the practitioner is interested in a specific feature (e.g., mean behavior or quantile) of the output, he should use more suited indices (e.g., the classical Sobol indices for the mean or the indices introduced in \cite{FKR13} for the quantile).

\section{Numerical applications}\label{snum}

\subsection{A flavor of the method applied on a toy model}\label{ssec:toy}

Let us  consider the quite simple linear model
\[
Y=\alpha X_1+ X_2,\;\; \alpha>0, 
\]
where $X_1$ has a Bernoulli distribution with success probability $0<p<1$ and $X_1$, $X_2$ are independent. Assume further that $X_2$ has a continuous distribution $F_2$ on $\R$ such that $\E[X_2]=\alpha p$ and 
with finite variance $\Var(X_2)=\alpha^2p(1-p)$. With these choices, the random variables  $\alpha X_1$ and $X_2$ share the same expectation and the same variance. Thus $X_1$ and $X_2$ have the same first order Sobol indices all equal to $1/2$. \\
We present a general closed formula to compute our new indices and show that in some particular cases an exact formula is available. Then we perform a simulation study in order to illustrate the Central Limit Theorem and analyse the practical behaviour of our estimators. 

\subsubsection{General closed formula}

On one hand, the distribution of $Y$ given $X_1=0$ and the distribution of $Y$ given $X_1=1$ are given by
\[
\begin{cases} 
\mathcal L(Y|X_1=0)=\mathcal L(X_2)\\
\mathcal L(Y|X_1=1)=\mathcal L(X_2+\alpha).\\
\end{cases}
\]
On the other hand, the conditional distribution of $Y$ given $X_2$ is 
$$\P\left(Y=\alpha+X_2\left|\right .X_2\right)=1-\P\left(Y=X_2|X_2\right)=p.$$
Hence, the distribution function of $Y$ 
is the mixture $pF_2(\cdot-\alpha)+(1-p)F_2(\cdot)$. 
Tedious computations lead to
\begin{align*}
S_{2,CVM}^{1}= 6 p(1-p)\int_\R(F_2(t)-F_2(t-\alpha))^2\left[(1-p)dF_2(t)+pdF_2(t-\alpha)\right]
%\label{H1_ex1}
\end{align*}
and
\begin{align*}
S_{2,CVM}^2=1-6p(1-p)\left[\frac{1}{2}-\int_\R F_2(t-\alpha)dF_2(t)\right] 
%\label{H2_ex1}
\end{align*}
(the normalizing factor being 1/6 as explained before). 

As $p$ goes to $0$ (and $\alpha$ goes to infinity), $(S_{2,CVM}^{1},S_{2,CVM}^{2})$ goes to $(0,1)$ while the two classical Sobol indices remain equal to $1/2$.
Our new indices shed lights on the fact  that, for small $p$,  $X_2$ has much more influence on $Y$ than $X_1$ which follows the intuition. This fact is  not detected by the classical Sobol indices. \\

Similarly we can compute the indices of order $q$ ($q\geqp 2$):
\begin{align*}
H_1^q=\alpha^q \left[p(1-p)^q+(-p)^q(1-p)\right] \quad \textrm{and} \quad H_2^q=\E[(X_2-\mu)^q].
\end{align*}

\tbf{Particular cases} (i) if $X_2$ is a centered Gaussian random variable with variance $\Var(X_2)=\alpha^2p(1-p)$, one can easily derive an explicit formula for $H_2^q$:
\begin{align*}
H_2^q&=\E[(X_2-m)^q]=\frac{q!}{2^{q/2}\cdot (q/2)!}\ind_{q\in 2\N^*}.
\end{align*}

(ii) if $X_2$ is uniformly distributed on $[0,b]$ with $b=2 \alpha \sqrt{3p(1-p)}$, one can easily derive an explicit formula for the indices introduced before:
\begin{align*}
S_{2,CVM}^{1}&=6p(1-p)\times\left(
\left(\frac{\alpha}{b}\right)^2\left(1-\frac 23 \frac{\alpha}{b}\right) \ind_{\alpha\leqp b}
+ \frac 13\ind_{\alpha> b}\right)\\
S_{2,CVM}^2&=1 -3p(1-p)\left(1-\left(\frac{b-\alpha}{b}\right)^2 \ind_{\alpha\leqp b}\right).
\end{align*}
Moreover, $H_2^q=\E[(X_2-\mu)^q]=
(b/2)^q/(q+1)\ind_{q\in 2\N^*}$.\\

(iii) if $X_2$ is exponentially distributed with mean $1/\lambda=\alpha \sqrt{p(1-p)}$, one can easily derive an explicit formula for the indices introduced before:
\begin{align*}
S_{2,CVM}^{1}=2p(1-p)(1-e^{-\lambda \alpha})^2 \quad \trm{and} \quad S_{2,CVM}^2=1 -3p(1-p)(1-e^{-\lambda \alpha}).
\end{align*}
Moreover,
$H_2^q=\E[(X_2-\mu)^q]=q!\lambda^{-q}/2$.

\subsubsection{Simulation study}
A numerical illustration with sample sizes $N$=100 and 500 is presented in Figures \ref{fig:simu_borgo_unif_1} and \ref{fig:simu_borgo_unif_2} (remind that in order to estimate both indices we compute $4N$ values of the output function). We consider the case where the random variable $X_2$ is uniformly distributed (for the other cases the simulations provide the same kind of results). We estimate the Cram\'er von Mises indices thanks to Equation \eqref{def:CvM_est} and renormalize it by the factor $1/6$ since the output has a continuous distribution. Then we estimate the limiting variance in \eqref{eq:var_xi} in order to provide asymptotic confidence intervals. In Figures  \ref{fig:simu_borgo_unif_1} and \ref{fig:simu_borgo_unif_2},  the blue line represents the true value of index $D^1_{2,CVM}$ (first row) or $D^2_{2,CVM}$ (second row). The red dashed line (resp. the red dashed line with +) represents the index estimation based on \eqref{def:CvM_est} (resp. the confidence interval). 
In the left column, $\alpha$ is fixed to 1/2 and $p$ varies while in the right one, $p$ is fixed to 1/4 and $\alpha$ varies.

\begin{figure}[p]
\centering
\includegraphics[width=17cm, height=17cm]{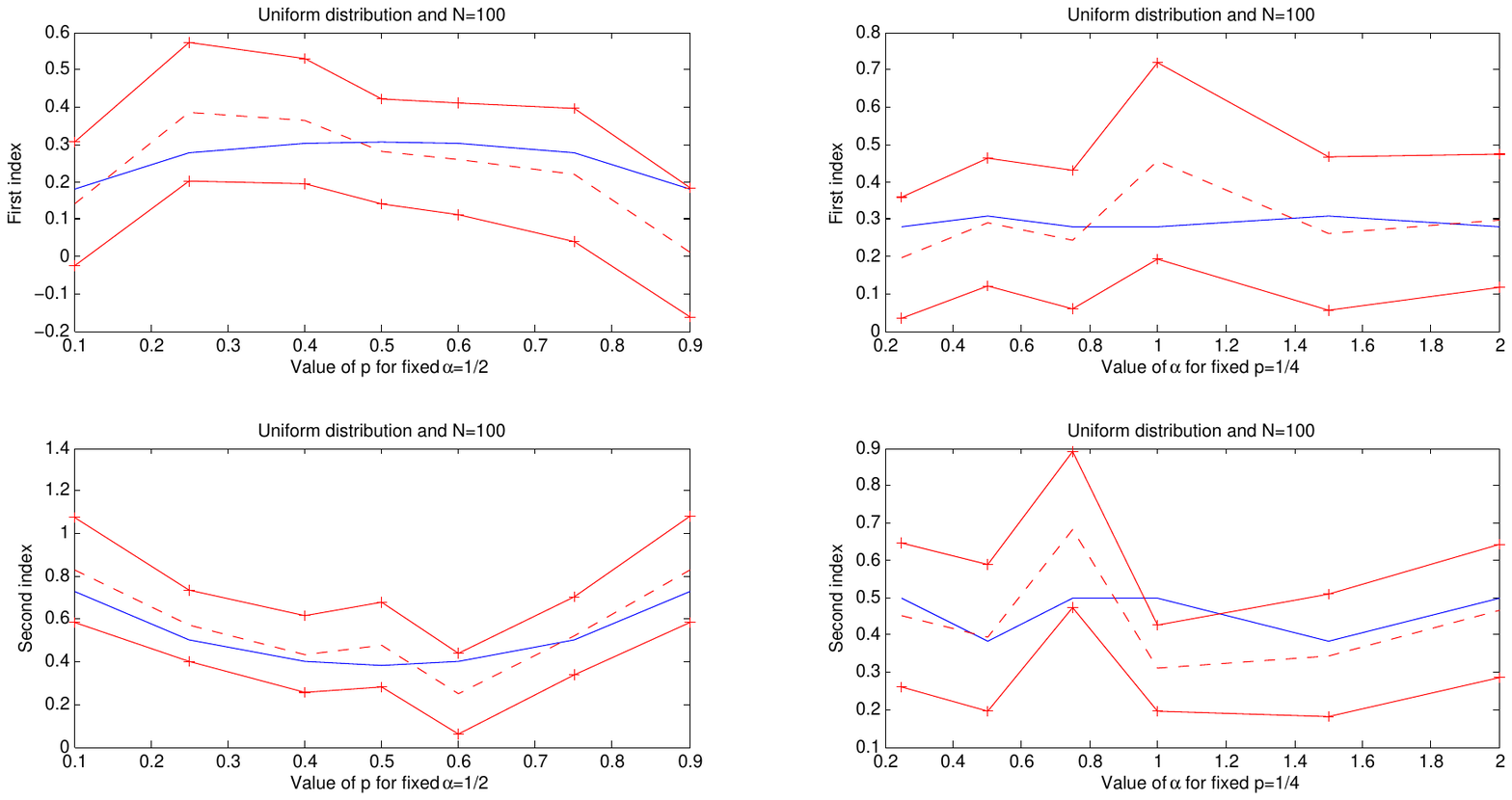}
\caption{Example 1 - $X_2$ uniformly distributed and $N$=100.}\label{fig:simu_borgo_unif_1}
\end{figure}

\begin{figure}[p]
\centering
\includegraphics[width=17cm, height=17cm]{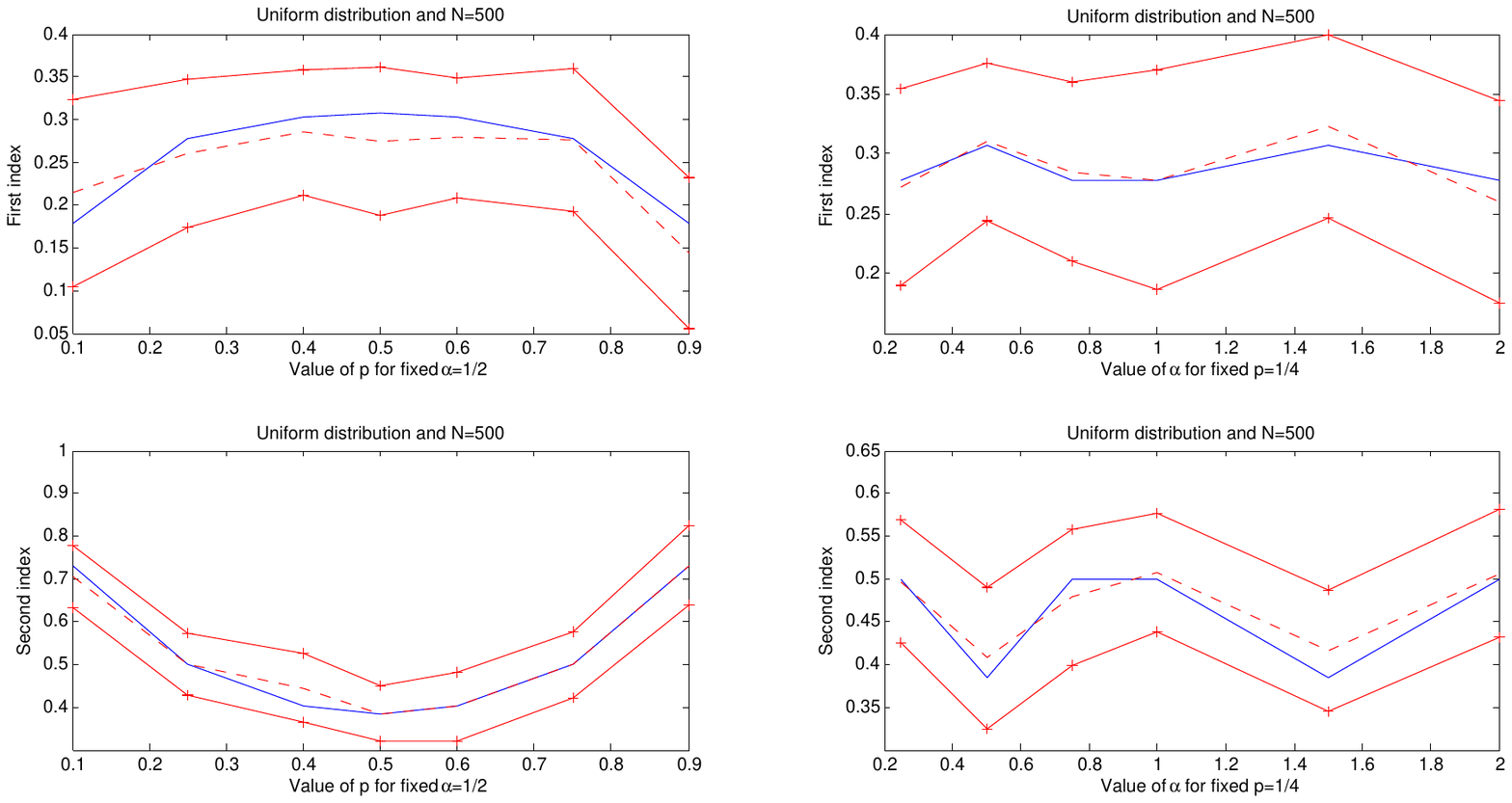}
\caption{Example 1 - $X_2$ uniformly distributed and $N$=500.}\label{fig:simu_borgo_unif_2}
\end{figure}

\subsection{A non linear model}

Now, let us consider the following non linear  model
\begin{equation}\label{eq:exp1}
Y=\exp\{ X_1+ 2X_2\}, 
\end{equation}
where $X_1$ and $X_2$ are independent standard Gaussian random variables. The distribution of $Y$ is log-normal and we can derive both its density and its distribution functions:
\[
f_Y(y)=\frac{1}{\sqrt{10\pi}y}e^{-(\ln y)^2/10}\1_{\R^+}(y)\quad \trm{and} \quad F_Y(y)=\Phi\left(\frac{\ln y}{\sqrt 5}\right) 
\]
where $\Phi$ stands for the distribution function of the standard Gaussian random variable. Its density function will be denoted by $f$ in the sequel. Then tedious computations lead to the Cramér von Mises indices $S_{2,CVM}^{1}$ and $S_{2,CVM}^{2}$.

\begin{prop}\label{prop:ex} Assume that $Y$ is defined by Equation \eqref{eq:exp1} then 
\begin{align*}
S_{2,CVM}^{1}=\frac{6}{\pi}\arctan 2-2\approx 0.1145
%\label{H1_ex2}
\end{align*}
and
\begin{align*}
S_{2,CVM}^2=\frac{6}{\pi}\arctan \sqrt{19}-2\approx 0.5693. 
%\label{H2_ex2}
\end{align*}
\end{prop}

\begin{rmk}
In this simple example, one can compute the indices of order $q$ ($q\geqp 2$):
\begin{align*}
H_1^q&=\E\left[(e^{X_1+2}-e^{5/2})^q\right]\quad \trm{and} \quad 
H_2^q=\E\left[(e^{2X_1+1/2}-e^{5/2})^q\right].
\end{align*}
\end{rmk}

The Sobol indices and their estimation based on the Pick-Freeze scheme with a sample of size $N$ are computed using equation (6) in \cite{janon2012asymptotic}. We also compute the Cramér von Mises indices and their estimation based on \eqref{def:CvM_est}. Moreover, we estimate the limiting variances in both cases (see equation \eqref{eq:var_xi} for the Cramér von Mises indices and equation (12) in \cite{janon2012asymptotic} for the Sobol indices) in order to provide confidence intervals. The results are presented in Table \ref{tab:ex2}. 

%\begin{table}[h]
%\begin{center}
%\begin{tabular}{cccccc}
	%\cline{3-6}
	%\multicolumn{2}{c}{} & \multicolumn{2}{c}{Cramér von Mises} & \multicolumn{2}{c}{Sobol indices}\\
	%\cline{3-6}
	%\multicolumn{2}{c}{} & $D^1_{2,CVM}$  & $D^2_{2,CVM}$ & $S^1$ & $S^2$\\
	%\cline{2-6}
		%\multicolumn{1}{c}{} & True values & 0.1145 & 0.5693 & 0.0118 & 0.3738\\
	%\hline
	    %$N=10^2$ & Est. values & 0.1943  & 0.5006 & 0.1962 & 0.1553\\
			         %&  CI  $5\%$     & [,]     & [,] & [,] & [,]\\
			%$N=10^3$ & Est. values & 0.1152  & 0.5576 & 0.0952 & 0.1085\\
			         %&  CI  $5\%$     & [,]     & [,] & [,] & [,]\\
			%$N=10^4$ & Est. values & 0.1055  & 0.5857 & 0.0092 & 0.3031\\
			         %&  CI  $5\%$     & [,]     & [,] & [,] & [,]\\
			%%$N=10^5$ & 0.0188  &  0.0947 & 0.0210 & 0.2201\\
			%\hline
%\end{tabular}
%\end{center}
%\caption{\label{tab:ex2} The Cramér von Mises indices and their estimation based on \eqref{def:CvM_est} in model \eqref{eq:exp1}}
%\end{table}

\begin{table}[h]
\begin{center}
\begin{tabular}{cccccc}
	\cline{3-6}
	\multicolumn{2}{c}{} & \multicolumn{2}{c}{Cramér von Mises} & \multicolumn{2}{c}{Sobol indices}\\
	\cline{3-6}
	\multicolumn{2}{c}{} & $D^1_{2,CVM}$  & $D^2_{2,CVM}$ & $S^1$ & $S^2$\\
	\cline{2-6}
		\multicolumn{1}{c}{} & True values & 0.1145 & 0.5693 & 0.0118 & 0.3738\\
	\hline
	    $N=10^2$ & Est. values & 0.1287 &   0.6097 & 0.0425 &  0.1954\\
			         &  CI  $5\%$     & [-0.0601,0.3175] & [0.4692,0.7503] & [0.0265,0.0585] &   [0.0430,0.3477]\\
			$N=10^3$ & Est. values & 0.1358 &   0.6007 & 0.1198 &   0.2345\\
			         &  CI  $5\%$     & [0.07861,0.19297] & [0.54897,0.65242] & [-0.5633,0.8030] & [0.1343,0.3347]\\
			$N=10^4$ & Est. values & 0.1166  & 0.5585 & 0.01685 &  0.26252\\
			         &  CI  $5\%$     & [0.09930,0.13382] &   [0.54150,0.57540] & [0.0010,0.0327] &   [-1.2744, 1.7994]\\
			\hline
\end{tabular}
\end{center}
\caption{\label{tab:ex2} Model \eqref{eq:exp1}. The Cramér von Mises and Sobol indices, their estimations based on \eqref{def:CvM_est} and (6) in \cite{janon2012asymptotic} and the associated $5\%$-confidence intervals.}
\end{table}

As a conclusion, with only $N=10^3$, the statistical method provides a precise estimation of the different indices. Moreover, in this example, the Sobol and Cramér von Mises indices give the same influence ranking of the two random inputs. Nevertheless, the estimation of the Cramér von Mises indices seems to be more efficient to give the true ranking.

\subsection{Application: The Giant Cell Arthritis Problem}

\paragraph{Context and goal}

\noi\\
In this subsection, we consider the realistic problem of management of suspected giant cell arthritis posed by Bunchbinder and Detsky in \cite{BD92}. More recently, this problem was also studied by Felli and Hazen \cite{FH04} and Borgonovo {\it et al.} \cite{BHP14}. As explained in \cite{BD92}, \enquote{giant cell arthritis (GCA) is a vasculitis of unknown etiology that affects large and medium sized vessels and occurs almost exclusively in patients 50 years or older}. This disease may lead to severe side effects (loss of visual accuity, fever, headache,...) whereas the risks of not treating it include the threat of blindness and major vessels occlusion. A patient with suspected GCA can receive a therapy based on Prednisone. Unfortunately, a treatment with high Prednisone doses may cause severe complications. Thus when confronted to a patient with suspected GCA, the clinician must adopt a strategy. There is a considerable literature on sensitivity analysis for these sorts
of models, based on the utility of learning a model input before choosing
a treatment strategy (see, e.g., \cite{FH98} and \cite{oakley09}).
In \cite{BD92}, the authors considered four different strategies:
\begin{itemize}
\item[A]: Treat none of the patients;
\item[B]: Proceed to the biopsy and treat all the positive patients;
\item[C]: Proceed to the biopsy and treat all the patients whatever their result;
\item[D]: Treat all the patients.
\end{itemize}
The clinician wants to adopt the strategy optimizing the patient outcomes measured in terms of utility. The reader is referred to \cite{NM53} for more details on the concept of utility. The basic idea is that a patient with perfect health is assigned a utility of 1 and the expected utility of the other patients (not perfectly healthy) is calculated subtracting some \enquote{disutilities} from this perfect score of 1. 
These strategies are represented in Figures \ref{fig:A} to \ref{fig:D} with the different inputs involved in the computation of the utilities.

%\begin{figure}[h]
%\centering
%\includegraphics[scale=0.3]{tree_A.pdf}
%\label{fig:A}
%\end{figure}
%\begin{figure}[h]
%\centering
%\includegraphics[scale=0.3]{tree_B.pdf}
%\label{fig:B}
%\end{figure}
%\begin{figure}[h]
%\centering
%\includegraphics[scale=0.3]{tree_C.pdf}
%\label{fig:C}
%\end{figure}
%\begin{figure}[h]
%\centering
%\includegraphics[scale=0.3]{tree_D.pdf}
%\label{fig:D}
%\end{figure}

\begin{figure}[h!]
\centering 
\includegraphics[width=10cm]{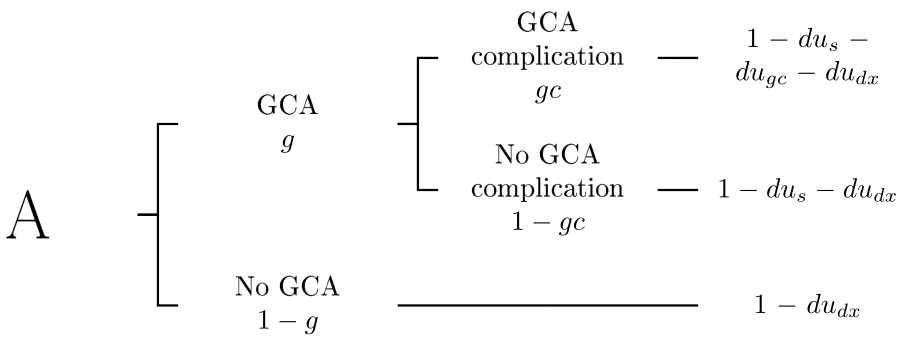}
\caption{The decision tree for the treat none alternative}
\label{fig:A}
\end{figure}

\begin{figure}[h!]
\centering 
\hspace{-2cm}
\includegraphics[width=12cm]{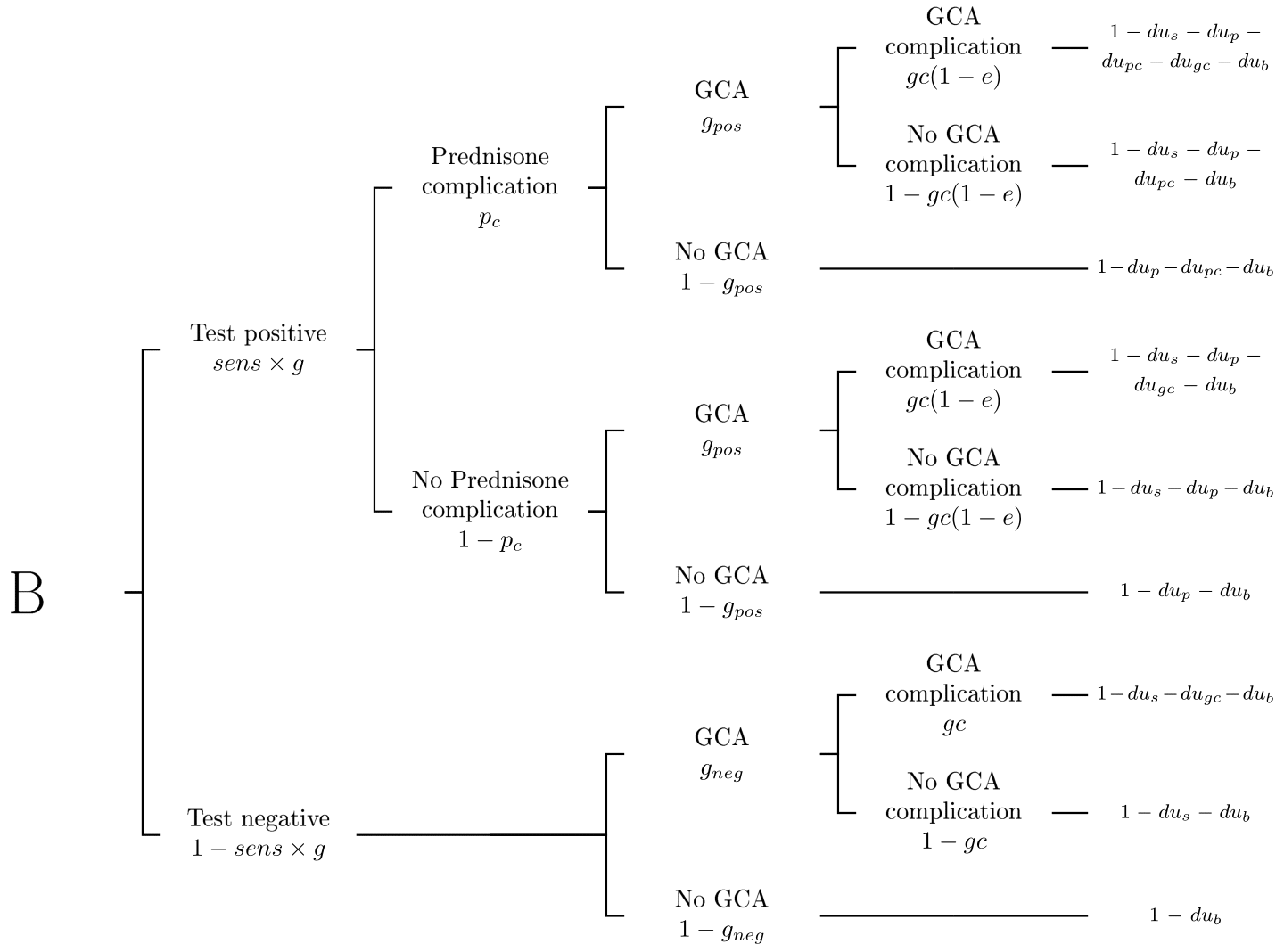}
\caption{The decision tree for the biopsy and the treat positive alternative}
\label{fig:B}
\end{figure}

\begin{figure}[h!]
\centering 
\hspace{-3cm}
\includegraphics[width=12cm]{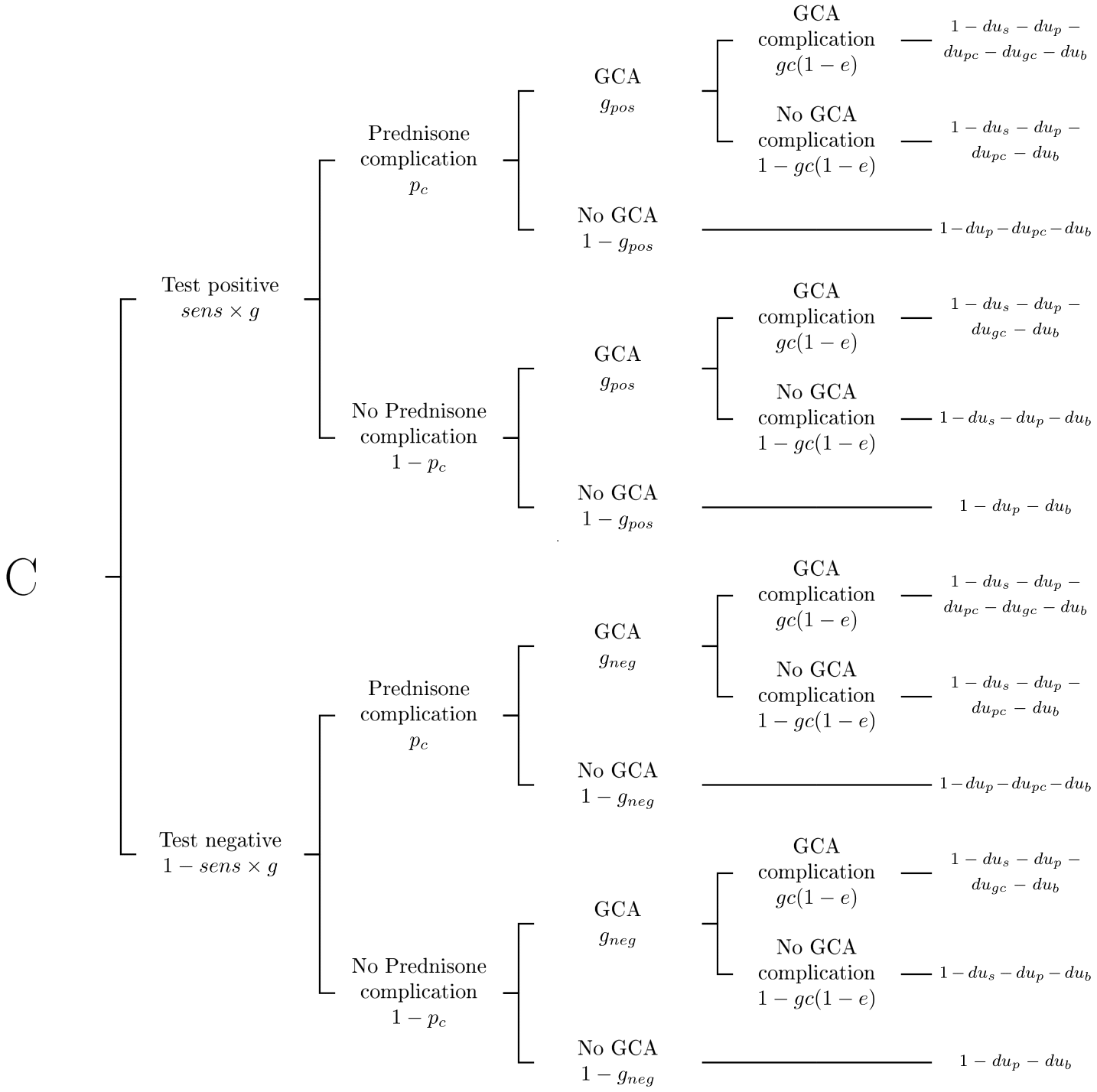}
\caption{The decision tree for the biopsy and the treat all alternative}
\label{fig:C}
\end{figure}

\begin{figure}[h!]
\centering 
\hspace{-1.5cm}
\includegraphics[width=10cm]{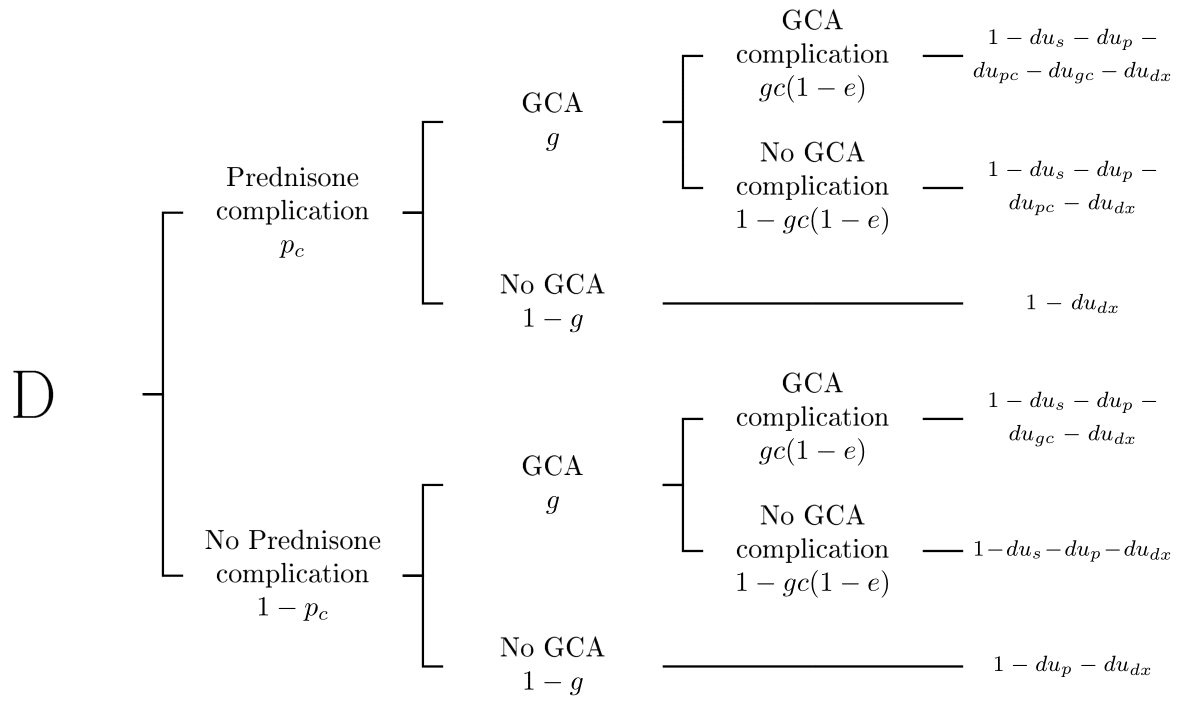}
\caption{The decision tree for the treat all alternative}
\label{fig:D}
\end{figure}

For example in strategy A (see Figure \ref{fig:A}), the utility of a patient having GCA and developing severe GCA complications is given by $1-d_s-du_{gc}-du_{dx}$. His entire sub-path is then
\[
g\times gc\times (1-d_s-du_{gc}-du_{dx}).
\]

\paragraph{The input parameters and the modelisation of the random ones}

\noi\\
As seen in Figures \ref{fig:A} to \ref{fig:D}, the different strategies involve input parameters like, e.g., the proportion $g$ of patients having GCA or the probability $gc$ for a patient to develop severe GCA complications (fixed at 0.8 as done in \cite{BD92}) or even the disutility associated to having GCA symptoms. Table \ref{tab:input} summarizes the input parameters involved. 

\begin{table}[h]
\begin{center}
\begin{tabular}{lllccll}
\hline
Fixed parameters & Symbols & Fixed value &  & &&\\
\hline
$\P$[having GCA] & $g$ & 0.8 & -- & -- & ~~~-- & ~~~-- \\
D(having symptoms of GCA) & $du_s$ & 0.12 & -- & -- & ~~~-- & ~~~-- \\
D(having a temporal artery biopsy) & $du_b$ & 0.005 & -- & -- & ~~~-- & ~~~-- \\
D(not knowing the true diagnosis) & $du_{dx}$ & 0.025 & -- & -- & ~~~-- & ~~~-- \\
\hline
\multicolumn{5}{l}{} & \multicolumn{2}{c}{Beta($\alpha$,$\beta$)}\\
\cline{6-7}
Uncertain parameters & Symbols & Base & Min. $m$ & Max. $M$ & ~~~$\alpha$ & ~~~$\beta$\\
\hline
$\P$[developing severe complications of GCA] & $gc$ & 0.3 & 0.05 & 0.5 & 4.179 & 11.011\\
$\P$[developing severe iatrogenic side effects] & $pc$ & 0.2 & 0.05 & 0.5 & 2.647 & 10.589\\
Efficacy of high dose Prednisone & $e$ & 0.9 & 0.8 & 1 & 27.787 & 3.087\\
Sensitivity of temporal artery biopsy & $sens$ & 0.83 & 0.6 & 1 & 7.554 & 1.547\\
D(major complication from GCA) & $du_{gc}$ & 0.8 & 0.3 & 0.9 & 27.454 & 6.864\\
D(Prednisone therapy) & $du_p$ & 0.08 & 0.03 & 0.2 & 4.555 & 52.380\\
D(major iatrogenic side effect) & $du_{pc}$ & 0.3 & 0.2 & 0.9 & 15.291 & 35.680\\
\hline
\end{tabular}
\end{center}
\caption{\label{tab:input} The data used by Buchbinder and Detsky \cite{BD92} in their analysis}
\end{table}

The values $\P[\cdot{} ]$ and $D(\cdot{} )$ refer respectively to the probability of an event and to the disutility associated with an event. The minimum and maximum values $m$ and $M$ depict each parameter's range for the sensitivity analysis. The base values are provided by a clinician expertise. 
The utilities of the different strategies when all the input parameters are set to their base value are summarized in Table \ref{tab:utilities}.

%\begin{center}
%\begin{tabular}{lc}
%\multicolumn{2}{l}{\tbf{Table 2 - The utilities of the different strategies }}\\
%\multicolumn{2}{l}{\tbf{when all the input parameters are set to }}\\
%\multicolumn{2}{l}{\tbf{their base value}}\\
%\hline
%Treatment alternative & Utilility\\
%\hline
%A Treat none & 0.6870\\
%B Biopsy and treat positive & 0.7575\\
%C Biopsy and treat all & 0.7398\\
%D Treat all & 0.7198\\
%\hline
%\end{tabular}
%\end{center}

\begin{table}[h]
\begin{center}
\begin{tabular}{lcc}
\hline
Treatment alternative & Utilility & Expectation\\
\hline
A Treat none & 0.6870 & 0.6991\\
B Biopsy and treat positive & 0.7575 & 0.7570\\
C Biopsy and treat all & 0.7398 & 0.7371\\
D Treat all & 0.7198 & 0.7171\\
\hline
\end{tabular}
\end{center}
\caption{\label{tab:utilities} The utilities of the different strategies when all the input parameters are set to their base value (second column) and their expectation when they are random (third column).}
\end{table}

The base value of some input parameters are reliable while the others are really uncertain that leads us to consider them as random. As a consequence, if $Y_A$, $Y_B$, $Y_C$ and $Y_D$ represent the outcomes corresponding to the four different strategies $A$ to $D$, the clinician aims to determine 
\begin{align}\
\max\{\E[Y_A],\E[Y_B],\E[Y_C],\E[Y_D]\}
\end{align}
with the uncertain model input presented in Table \ref{tab:input}. A sensitivity analysis is then performed to determine the most influential input variables on the outcome.\\

\noi
As done in \cite{FH04,BHP14}, all the random inputs will be independently based on Beta distributions. The Beta density parameters corresponding to each random input are determined by fitting the base value as their mean and capturing 95$\%$ of the probability mass in the range defined by the minimum and maximum. The remaining 5$\%$ will be equally distributed to either side of this range if possible. Concretely, each random input will be distributed as
\[
Z\1_{m\leqp Z< M}+U\1_{m> Z}+V\1_{Z\geqp M}
\]
where $Z$, $U$ and $V$ are independent random variables. $Z$ is Beta distributed with parameters $(\alpha,\beta)$. $U$ and $V$ are uniform random variables on $[0,m]$ and $[M,1]$ respectively. 

\paragraph{Sensitivity analysis}

\noi\\
As already mentioned, the clinician wants to determine the highest utility. In \cite{BB13}, the authors then consider the highest utility as output and lead a sensitivity analysis to determine the input having the largest influence on this output. Since we are able to treat multivariate outputs, we consider a more general framework in this paper: the output is the four-dimensional random variable $Y=(Y_A, Y_B, Y_C, Y_D)$ where $Y_S$ represents the outcome corresponding to strategy $S$. \\

We compare three different methodologies and indices. First, we consider the Sobol indices introduced in \cite{GJKL14} (Multivariate). Second, we consider the indices constructed in this paper, based on the Cramér von Mises distance and estimated by the ratios of the numerator estimator \eqref{def:CvM_est} and the denominator estimator \eqref{def:CvM_est_den}. Third, we consider the index presented in \cite{BB13} and named $\beta$ defined by
\[
\beta_i=\E[\sup_{y\in\mathcal{Y}} \{|F_Y(y)-F_{Y|X_i}(y)|\}].
\]
Then we use the estimator given in \cite[Table 1]{BHP14} adapted to the multivariate case that is based on the tedious and costly estimation of conditional expectations.

\paragraph{Results}

\noi\\
Table \ref{tab:results_borgo} summarizes the sensitivity measures of the seven random inputs on the multivariate output with the three different methodologies while Table \ref{tab:results_borgo_rank} presents the associated ranks. It is worth mentioning that the same total sample size has been used to compare properly the three methodologies. 

%\begin{table}[p]
%\begin{center}
%\begin{tabular}{c|ccccccc|c}
%\hline
%Sens. meas. & $gc$ & $pc$ & $e$ & $sens$ & $du_{gc}$ & $du_{p}$ & $du_{pc}$ & Rank\\
%\hline
%Multivariate &  &  &  &  &  &  &  & 1236475\\
%Borgonovo {\it et al.}  &  &  &  &  &  &  &  & 1627354\\
%Cramér von Mises &  &  &  &  &  &  &  & 1627354\\
%\hline
%\end{tabular}
%\end{center}
%\caption{\label{tab:results_borgo} Sensitivity measures}
%\end{table}

\begin{table}[h]
\begin{center}
\begin{tabular}{|c|c|c|c|c|c|c|c|c||c|}
\cline{2-10}
\multicolumn{1}{c|}{} & Sensitivity meas. & 1 & 2 & 3 & 4 & 5 &  6 & 7 & Cputime \\
\hline
& Multivariate &  0.3690  &  0.0193  &  0.0105  & -0.0821  & -0.0617   & 0.1150  & -0.0751 & 0.0624\\
$N=10^2$ & Borgonovo {\it et al.} & 0.1195  &  0.1047  &  0.1064  &  0.1022   & 0.1046  &  0.1063  &  0.1027  & 1.5132\\
& Cramér von Mises &  0.3496  &  0.0745  &  0.0206  & -0.0010 &   0.0084 &   0.1042 &   0.0105 & 0.9048\\
\hline
\hline
& Multivariate &  0.4024 &   0.1201 &   0.0516 &  -0.0190  & -0.0043  &  0.2403  &  0.0093 &0.0156\\
$N=10^3$ & Borgonovo {\it et al.} & 0.1788 &   0.1192  &  0.1009  &  0.1007  &  0.1044  &  0.1195  &  0.1028 & 57.8452\\
& Cramér von Mises & 0.3494 &   0.0750  &  0.0209  & -0.0008  &  0.0086  &  0.1045  &  0.0109 & 10.1089\\
\hline
\hline
& Multivariate &  0.3828 &   0.1333 &   0.0618 &  -0.0016  &  0.0100 &   0.3182   & 0.0217  & 0.0312  \\
$N=10^4$ & Borgonovo {\it et al.} & 0.3842  &  0.1572 &   0.1033 &   0.0930 &   0.0986   & 0.1775  &  0.1061 &  5.1988 $10^3$\\
& Cramér von Mises &  0.3494    &0.0775  &  0.0232 &   0.0011 &   0.0108  &  0.1056 &   0.0124 & 436.8028   \\
\hline
\end{tabular}
\end{center}
\caption{\label{tab:results_borgo} Sensitivity measures. The estimation of the Cramér von Mises indices is the ratio of \eqref{def:CvM_est} and \eqref{def:CvM_est_den}.}
\end{table}

\begin{table}[h]
\begin{center}
\begin{tabular}{ccc}
\cline{2-3}
\multicolumn{1}{c}{} & Sensitivity meas. & Ranking\\
\hline
 & Multivariate &   1     6     2     3     5     7     4\\
$N=10^2$ & Borgonovo {\it et al.} & 1     3     6     2     5     7     4\\
& Cramér von Mises & 1     6     2     3     7     5     4\\
\hline
 & Multivariate &   1     6     2     3     7     5     4\\
$N=10^3$ & Borgonovo {\it et al.} & 1     6     2     5     7     3     4\\
& Cramér von Mises & 1     6     2     3     7     5     4\\
\hline
 & Multivariate & 1     6     2     3     7     5     4\\
$N=10^4$ & Borgonovo {\it et al.} & 1     6     2     7     3     5     4\\
& Cramér von Mises & 1     6     2     3     7     5     4\\
\hline
\end{tabular}
\end{center}
\caption{\label{tab:results_borgo_rank} Ranks. The estimation of the Cramér von Mises indices is the ratio of \eqref{def:CvM_est} and \eqref{def:CvM_est_den}.}
\end{table}

As a conclusion, in this example, unlike the indices defined by Borgonovo  {\it et al.}, the multivariate sensitivity indices and the Cramér von Mises indices provide the same ranking. The main advantage of the Cramér von Mises 
sensitivity methodology with respect to the one of Borgonovo  {\it et al.} is that one can use the Pick and Freeze estimation scheme which provides an accurate estimation (see \eqref{def:CvM_est}) simple to implement.

Notice that in \cite{BHP14}, the authors study a slightly different model that explains the numerical differences between their results and the ones of the present paper. 
Furthermore, they perform a sensitivity analysis on the best alternative with the greater mean instead of considering the multivariate output.

\section{Conclusion}

In this paper, we first study the asymptotic properties of the multiple Pick and Freeze scheme proposed by Owen {\it et al.} for the estimation of higher order Sobol indices. 
This index has several drawbacks that lead us to propose a new natural index based on the Cram\'er von Mises distance between the distribution of the output $Y$ and the conditional law when an input is fixed. 
This new index contains all the distributional information, is naturally defined for multivariate outputs and provides a rigorous sharper way for a fast screening of complex computer codes. Furthermore, our approach is generic and may be 
extended and implemented for general outputs (vectorial, valued on a manifold, functional, ...).   
Concerning its estimation, we show that surprisingly a Pick and Freeze scheme is also available for the estimation procedure and prove that it is efficient in a theoretical point of view as well as in a practical one. More precisely, we establish a Central Limit Theorem that confirms the good statistical properties of our estimator and allows us to build confidence intervals.
Furthermore, the estimation is well working with moderate sample sizes as shown in toy examples. Finally, the performance of the method is proven on a real data example. 

%For example, it gives good results on 
%the real data example studied in the last section of the paper.

%\paragraph{Practical advices} 
%
%\noi\\
%In a general setting, for all the nice properties of the Cramér von Mises indices and their efficient estimation easy to implement, we recommend to use the Cramér von Mises indices. Nevertheless, when the practitioner is interested in a specific feature (e.g., mean behavior or quantile) of the output, he should use more suited indices (e.g., the classical Sobol indices for the mean or the indices introduced in \cite{FKR13} for the quantile).

\section*{Acknowledgement} The authors are greatly indebted to the referees for their
fruitful and detailed suggestions or comments which permit us to greatly improve our paper.\\
Part of this research was conducted within the frame of the Chair in Applied Mathematics OQUAIDO and   the ANR project PEPITO (ANR-14-CE23-0011).\\

\section{Proofs}\label{sec:proofs}

\subsection{Proof of Theorem \ref{th:TCL}}

\begin{proof}[Proof of Theorem \ref{th:TCL}]
The consistency follows from a straightforward application of the strong law of large numbers. The asymptotic normality is derived by two successive applications of the delta method  \cite{van2000asymptotic} . \\

(1) Let $W_j^1\defeq (Y^{v,1}_{j},\ldots, Y^{v,p}_{j})^T$ ($j=1,\ldots, N$) and $g^1$ be the mapping from $\R^{p}$ to $\R^{p}$ whose $l$-th coordinate is given by 
\[
g^1_l(x_1,\ldots, x_p)=\binom{p}{l}^{-1}\sum_{\begin{matrix}k_1<\ldots < k_l \\ k_{i}\in I_p,i=1,\ldots,l\end{matrix}}\left(\prod_{i=1}^{l} x_{k_i}\right).
\]
Then $(W_j^1)_{j=1,\ldots, N}$ is an i.i.d. sample distributed as $W^1\defeq (Y^{v,1},\ldots, Y^{v,p})^T$.

Let $\Sigma^1$ be the covariance matrix of $W_j^1$. Clearly, one has $\Sigma_{ii}^1=\Var(Y)$ for $i\in I_p$ while for $i\neq j,$ $\Sigma_{ij}^1=\Cov(Y^{v,i},Y^{v,j})=\Cov(Y,Y^{v,2})$. The multidimensional Central Limit Theorem gives that

\begin{equation*}
\sqrt{N}\left(\frac1N \sum_{j=1}^N W_j^1-m\right)\overset{\mathcal{L}}{\underset{N\to\infty}{\rightarrow}} \mathcal{N}_{p}\left(0,\Sigma^1\right),
\end{equation*}
where $m:=(\E[Y],\ldots, \E[Y])^T$. We then apply the so-called delta method to $W^1$ and $g^1$ so that
\begin{equation*}
\sqrt{N}\left(g^1\left(\overline{W}^1_N\right)-g^1\left(\E\left[W^1\right]\right)\right)\overset{\mathcal{L}}{\underset{N\to\infty}{\rightarrow}} \mathcal{N}\left(0,J_{g^1}\left(\E\left[W^1\right]\right)\Sigma^1 J_{g^1}\left(\E\left[W^1\right]\right)^T\right)
\end{equation*}
where $J_{g^1}\left(\E\left[W^1\right]\right)$ is the Jacobian of $g^1$ at point $\E\left[W^1\right]$. Notice that for $i\in I_p$ and $k\in I_p$,
\[
\frac{\partial g_l^1}{\partial x_k} \left(\E\left[W^1\right]\right)=\frac{\binom{p-1}{l-1}}{\binom{p}{l}} m^{l-1}=\frac{l}{p}\E[Y]^{l-1}=:a_l.
\]

Thus $\Sigma^2:=J_{g^1}\left(\E\left[W^1\right]\right)\Sigma^1 J_{g^1}\left(\E\left[W^1\right]\right)^T$ is given by

\[
\Sigma^2_{ij}=pa_ia_j \left(\Sigma_{11}^1+(p-1)\Sigma_{12}^1\right).
\]

(2) Now consider $W_j^2\defeq (P^{v,1}_{j},\ldots P^{v,p}_{j})^T$ ($j=1,\ldots, N$) and $g^2$ the mapping from $\R^{p}$ to $\R$ defined by 
\[
g^2(y_1,\ldots, y_p)=\sum_{l=0}^{p}\binom{p}{l}(-1)^{p-l} y_1^{p-l} y_{l}.
\]
Then $(W_j^2)_{j=1,\ldots, N}$ is an i.i.d. sample distributed as $W^2\defeq (P^{v,1},\ldots P^{v,p})^T$.

We apply once again the delta method to $W^2$ so that
\begin{equation*}
\sqrt{N}\left(g^2\left(\overline{W}^2_N\right)-g^2\left(\E\left[W^2\right]\right)\right)\overset{\mathcal{L}}{\underset{N\to\infty}{\rightarrow}} \mathcal{N}\left(0,J_{g^2}\left(\E\left[W^2\right]\right)\Sigma^2 J_{g^2}\left(\E\left[W^2\right]\right)^T\right)
\end{equation*}
where $J_{g^2}\left(\E\left[W^2\right]\right)$ is the Jacobian of $g^2$ at point $\E\left[W^2\right]$. Notice that for $k\in I_p$,
\begin{align*}
\frac{\partial g^2}{\partial y_1} \left(\E\left[W^2\right]\right)&= (-1)^{p-1}p(p-1) \E[Y]^{p-1} \\
&+\sum_{l=2}^{p-1}\binom{p}{l}(-1)^{p-l} (p-l) \E[Y]^{p-l-1} \E\left[\prod_{i=1}^l Y^{v,i}\right]
\end{align*}
and
\begin{align*}
\frac{\partial g^2}{\partial y_l} \left(\E\left[W^2\right]\right)&= \binom{p}{l}(-1)^{p-l}\E[Y]^{p-l}.
\end{align*}

Thus the limiting variance is 
\[
\sigma^2:=J_{g^2}\left(\E\left[W^2\right]\right)\Sigma^2 J_{g^2}\left(\E\left[W^2\right]\right)^T=p \left(\Sigma_{11}^1+(p-1)\Sigma_{12}^1\right)\left(\sum_{i=1}^p a_ib_i\right)^2,
\]
where $b_i$ is the $i$-th coordinate of $\nabla g^2\left(\E\left[W^2\right]\right)$. 
\end{proof}

\subsection{An auxiliary result and the proofs of the results of Section \ref{sec:Cramér}}

%\begin{proof}[Proof of Proposition \ref{prop:prop}]
%Notice that for any $t$,
%\[
%\E_{X_v}\left[\left(F(t)-F^{v}(t)\right)^{2}\right]=\Var(F^{v}(t))\leqp \frac 14.
%\]
%Indeed, $0\leqp F^{v}(t)=\E\left[\1_{\{Z\leqp t\}}|X_{v}\right] \leqp 1$ and the variance of any random variable valued in $[a,b]$ is less or equal than $\left(\frac{b-a}{2}\right)^2$.
%Integrating on $t$,  the first point of the last proposition gives an upper bound for $S_{2,CVM}^{v}$. The second point follows.
%\end{proof}

\begin{lemma}\label{lem:cv}
Let $G$ and $H$ be two  
measurable functions.
Let $(U_j)_{j\in I_N}$ and $(V_k)_{k\in I_N}$ be two independent samples of i.i.d. random variables. Assume that $G(U_1,V_1)$ and $H(U_1,U_2,V_1)$ are both integrable and centered. 
We define $S_N$ and $T_N$ by 
\[
S_N=\frac{1}{N^2} \sum_{j,k=1}^N G(U_j,V_k) \quad \trm{and} \quad T_N=\frac{1}{N^3} \sum_{i,j,k=1}^N H(U_i,U_j,V_k).
\]
Then $S_N$ and $T_N$ converge a.s. to 0 as $N$ goes to infinity.
\end{lemma}

\begin{proof}[Proof of Lemma \ref{lem:cv}]
Notice that if $\E[S_N^4]=O\left(\frac{1}{N^2}\right)$ then by Borel-Cantelli lemma, $S_N$ converges a.s. to 0. Now,

\begin{align*}
\E[S_N^4]&=\frac{1}{N^8} \sum \E[G(U_{i_1},V_{j_1})G(U_{i_2},V_{j_2})G(U_{i_3},V_{j_3})G(U_{i_4},V_{j_4})]
\end{align*}

where the sum is taken over all the indices $i_1$, $i_2$, $i_3$, $i_4$, $j_1$, $j_2$, $j_3$, $j_4$ from 1 to $N$. The only cases leading to terms in 
$O\left(\frac{1}{N}\right)$ or even in $O\left(1\right)$ appear when we sum over indices that are all different except two $i$'s or two $j$'s or over 
indices that are all different. Nevertheless, in those cases, at least one term of the form $\E[G(U_{i},V_{j})]$ appears. Since the function $G$ is centered, those cases are then discarded.\\\\
The proof of the result concerning $T_N$ follows the same tracks.
\end{proof}

\begin{proof}[Proof of Corollary \ref{cor:cons_N2_est}]
The proof is based on Lemma \ref{lem:cv}. First, we define $Z_j=\left(Z_j^{v,1},Z_j^{v,2}\right)$,
\begin{align*}
&G(Z_j,W_k)=\1_{\{Z_j^{v,1}\leqp W_k\}} \1_{\{Z_j^{v,2}\leqp W_k\}},\\
&F(Z_j,W_k)=\frac{1}{2}\left(\1_{\{Z_j^{v,1}\leqp W_k\}}+ \1_{\{Z_j^{v,2}\leqp W_k\}}\right),\\
&H(Z_i,Z_j,W_k)=F(Z_i,W_k)F(Z_j,W_k).
\end{align*}

Second, we proceed to the following decomposition
\begin{align*}
&\widehat N_{2,CVM}^{v}=\frac 1N \sum_{k=1}^N \left\{\frac 1N \sum_{j=1}^N  \1_{\{Z_j^{v,1}\leqp W_k\}} \1_{\{Z_j^{v,2}\leqp W_k\}}-\left[\frac {1}{2N} \sum_{j=1}^N  \left(\1_{\{Z_j^{v,1}\leqp W_k\}}+ \1_{\{Z_j^{v,2}\leqp W_k}\}\right)\right]^2\right\}\\
&=\frac{1}{N^2} \sum_{j,k=1}^N \1_{\{Z_j^{v,1}\leqp W_k\}} \1_{\{Z_j^{v,2}\leqp W_k\}}-\frac {1}{4N^3} \sum_{i,j,k=1}^N  \left(\1_{\{Z_i^{v,1}\leqp W_k\}}+ \1_{\{Z_i^{v,2}\leqp W_k\}}\right)\left(\1_{\{Z_j^{v,1}\leqp W_k\}}+ \1_{\{Z_j^{v,2}\leqp W_k\}}\right)\\
&=\frac{1}{N^2} \sum_{j,k=1}^N G(Z_j,W_k)-\frac {1}{N^3} \sum_{i,j,k=1}^N  H(Z_i,Z_j,W_k)\\
&=\frac{1}{N^2} \sum_{j,k=1}^N \left\{G(Z_j,W_k)-\E[G(Z_j,W_k)]\right\}-\frac {1}{N^3} \sum_{i,j,k=1}^N  \left\{H(Z_i,Z_j,W_k)-\E[H(Z_i,Z_j,W_k)]\right\}\\
&\trm{~~~~}+ \frac {1}{N^2} \sum_{j,k=1}^N \E[G(Z_j,W_k)]- \frac {1}{N^3} \sum_{i,j,k=1}^N  \E[H(Z_i,Z_j,W_k)]\\
&=\frac{1}{N^2} \sum_{j,k=1}^N \left\{G(Z_j,W_k)-\E[G(Z_j,W_k)]\right\}-\frac {1}{N^3} \sum_{i,j,k=1}^N  \left\{H(Z_i,Z_j,W_k)-\E[H(Z_i,Z_j,W_k)]\right\}\\
&\trm{~~~~}+ \E[G(Z_1,W_1)]- \left(1-\frac {1}{N}\right) \E[H(Z_1,Z_2,W_1)]-\frac {1}{N} \E[H(Z_1,Z_1,W_1)].\\
\end{align*}
The two first sums converge almost surely to 0 by Lemma \ref{lem:cv}. The remaining term goes to $\E[G(Z_1,W_1)]- \E[H(Z_1,Z_2,W_1)]$ as $N$ goes to infinity.\\\\
It remains to show that $N_{2,CVM}^{v}=\E[G(Z_1,W_1)]- \E[H(Z_1,Z_2,W_1)]$. On the one hand,
\begin{align*}
N_{2,CVM}^{v}&=\int_{\R}\E[(F(t)-F^v(t))^2]dF(t)=\E[H_v^2(W_1)]\\
&=\E[\Cov(\1_{\{Z_1^{v,1}\leqp W_1\}} ,\1_{\{Z_1^{v,2}\leqp W_1\}})]\\
&=\E_W[\E_Z[\1_{\{Z_1^{v,1}\leqp W_1\}} \1_{\{Z_1^{v,2}\leqp W_1\}}]-\E_Z[\1_{\{Z_1^{v,1}\leqp W_1\}}]^2].\\
\end{align*}
On the other hand,
\begin{align*}
&\E[G(Z_1,W_1)]- \E[H(Z_1,Z_2,W_1)]\\
&=\E[\1_{\{Z_1^{v,1}\leqp W_1\}} \1_{\{Z_1^{v,2}\leqp W_1\}}]- \frac{1}{4}\E[(\1_{\{Z_1^{v,1}\leqp W_1\}}+ \1_{\{Z_1^{v,2}\leqp W_1\}})(\1_{\{Z_2^{v,1}\leqp W_1\}}+ \1_{\{Z_2^{v,2}\leqp W_1\}})]\\
&=\E_W[\E_Z[\1_{\{Z_1^{v,1}\leqp W_1\}} \1_{\{Z_1^{v,2}\leqp W_1\}}]]-\E[\1_{\{Z_1^{v,1}\leqp W_1\}}\1_{\{Z_2^{v,2}\leqp W_1\}}]\\
&=\E_W[\E_Z[\1_{\{Z_1^{v,1}\leqp W_1\}} \1_{\{Z_1^{v,2}\leqp W_1\}}]]-\E[\E[\1_{\{Z_1^{v,1}\leqp W_1\}}\1_{\{Z_2^{v,2}\leqp W_1\}}\vert W_1]]\\
&=\E_W[\E_Z[\1_{\{Z_1^{v,1}\leqp W_1\}} \1_{\{Z_1^{v,2}\leqp W_1\}}]]-\E[\E[\1_{\{Z_1^{v,1}\leqp W_1\}}\vert W_1]\E[\1_{\{Z_2^{v,2}\leqp W_1\}}\vert W_1]]\\
&=\E_W[\E_Z[\1_{\{Z_1^{v,1}\leqp W_1\}} \1_{\{Z_1^{v,2}\leqp W_1\}}]]-\E[\E[\1_{\{Z_1^{v,1}\leqp W_1\}}\vert W_1]]\E[\E[\1_{\{Z_2^{v,2}\leqp W_1\}}\vert W_1]]\\
&=\E_W[\E_Z[\1_{\{Z_1^{v,1}\leqp W_1\}} \1_{\{Z_1^{v,2}\leqp W_1\}}]]-\E[\1_{\{Z_1^{v,1}\leqp W_1\}}]\E[\1_{\{Z_2^{v,2}\leqp W_1\}}]\\
&=\E_W[\E_Z[\1_{\{Z_1^{v,1}\leqp W_1\}} \1_{\{Z_1^{v,2}\leqp W_1\}}]]-\E[\1_{\{Z_1^{v,1}\leqp W_1\}}]^2\\
&=\E_W[\E_Z[\1_{\{Z_1^{v,1}\leqp W_1\}} \1_{\{Z_1^{v,2}\leqp W_1\}}]-\E_Z[\1_{\{Z_1^{v,1}\leqp W_1\}}]^2]
\end{align*}
that completes the proof.
\end{proof}

\begin{proof}[Proof of Theorem \ref{th:as_norm_N2_est}]
We define for $t\in \R$,
\begin{align*}
&\mathbb{G}_N^{1,2}(t,t)=\frac 1N \sum_{j=1}^N  \1_{\{Z_j^{v,1}\leqp t\}} \1_{\{Z_j^{v,2}\leqp t\}},\\
&\mathbb{G}_N^i(t)=\frac 1N \sum_{j=1}^N  \1_{\{Z_j^{v,i}\leqp t\}},\; i=1,2,\\
&\mathbb{F}_N(t)=\frac 1N \sum_{k=1}^N  \1_{\{W_k\leqp t\}}
\end{align*}
and we rewrite $\widehat N_{2,CVM}^{v}$ as  a regular function depending on the four empirical processes defined above:
\begin{align*}
\widehat N_{2,CVM}^{v}&= \int \left[\mathbb{G}_N^{1,2}-\left(\frac{\mathbb{G}_N^1+\mathbb{G}_N^2}{2}\right)^2\right]d\mathbb{F}_N.\\
\end{align*}
By Donsker's theorem,
\[
\sqrt{N}\left(\mathbb{G}_N^{1,2}-\widetilde{G},\mathbb{G}_N^1-F,\mathbb{G}_N^2-F,\mathbb{F}_N-F\right)\overset{\mathcal{L}}{\underset{N\to\infty}{\rightarrow}}\mathbb{G} =(\mathbb{G}_1,\mathbb{G}_2,\mathbb{G}_3,\mathbb{G}_4)
\]
where $G(t,s)=\P\left(Z^{v,1}\leqp t,\; Z^{v,2}\leqp s\right)$, $\widetilde{G}(t)=G(t,t)$ and $\mathbb{G}$ is a centered Gaussian process of dimension 4 with covariance function
defined by
\[
\Pi(t,s)=\E\left(A_tA_s^T\right)-\E\left(A_t\right)\E\left(A_s\right)^T,\;  \textrm{for} \quad (t,s) \in \R^2
\]
and $A_t:=\left(\1_{\{Z^{v,1}\leqp t\}} \1_{\{Z^{v,2}\leqp t\}},\1_{\{Z^{v,1}\leqp t\}}, \1_{\{Z^{v,2}\leqp t\}},\1_{\{W\leqp t\}}\right)^T$.\\

Since these processes are càd-làg functions of bounded variation, we introduce the maps $\psi_1,\; \psi_2:BV_1[-\infty,+\infty]^2\mapsto \R$ and 
$\Psi:BV_1[-\infty,+\infty]^4\mapsto \R$ defined by 
\[
\psi_i(F_1,F_2)=\int (F_1)^idF_2,\; i=1,2 \quad  \trm{and} \quad \Psi(F_1,F_2,F_3,F_4)=\psi_1(F_1,F_4)-\psi_2\left(\frac{F_2+F_3}{2},F_4\right),
\]
where $BV_M[a,b] $ is the set of càd-làg functions of variation bounded by $M$. Hence,
\begin{align*}
\widehat N_{2,CVM}^{v}&= \Psi\left(\mathbb{G}_N^{1,2},\mathbb{G}_N^1,\mathbb{G}_N^2,\mathbb{F}_N\right),
\end{align*}

Now using the chain rule 20.9 and Lemma 20.10 in \cite{van2000asymptotic}, the map $\Psi$ is Hadamard-differentiable from the domain $BV_1[-\infty,+\infty]^4$ into $\R$ whose derivative is given by
\[
(h_1,h_2,h_3,h_4)\mapsto D\psi_1{(F_1,F_4)}(h_1,h_4)-D\psi_2{\left(\frac{F_2+F_3}{2},F_4\right)}\left(\frac{h_2+h_3}{2},h_4\right)
\]
where the derivative of $\psi_i$ are given by Lemma 20.10:
\[
(h_1,h_2)\mapsto h_2\varphi_i \circ F_1\vert_{-\infty}^{+\infty}-\int h_{2-}d\varphi_i \circ F_1+\int \varphi'_i(F_1)h_1dF_2
\]
with $\varphi_i(x)=x^i$ and $h_-$ is the left-continuous version of a càd-làg function $h$.\\

Applying the functional delta method 20.8 in \cite{van2000asymptotic} we get the weak convergence  of $\sqrt{N}\left(\widehat N_{2,CVM}^{v}- N_{2,CVM}^{v}\right)$ to the following limit distribution
\[
\int \mathbb{G}_{4-}d(F^2-\widetilde{G})+\int \mathbb{G}_1 dF-\int F(\mathbb{G}_2+\mathbb{G}_3)dF.
\]
Since the map $\Psi$ is continuous on the whole space $BV_1[-\infty,+\infty]^4$, the delta method in its stronger form 20.8 in \cite{van2000asymptotic} implies that the limit variable is the limit in distribution of the sequence
\begin{align*}
&D\Psi{(\widetilde{G},F,F,F)}\left(\sqrt{N}\left(\mathbb{G}_N^{1,2}-\widetilde{G},\mathbb{G}_N^1-F,\mathbb{G}_N^2-F,\mathbb{F}_N-F\right)\right)\\
&=\sqrt{N} \left[\int \left(\mathbb{F}_N-F\right)_-d\left(F^2-\widetilde G)\right)+\int \left(\mathbb{G}_N^{1,2}-\widetilde{G}-F\left(\mathbb{G}_N^{1}+\mathbb{G}_N^{2}-2F\right)\right)dF\right].
\end{align*}
We define
\begin{align*}
U&:=\int \1_{\{W< t\}} d(F^2(t)-\widetilde G(t)= \widetilde  G(W)-F(W)^2,\\
V&:=\int \left[\1_{\{Z^{v,1}\leqp t\}}\1_{\{Z^{v,2}\leqp t\}}-\left(\1_{\{Z^{v,1}\leqp t\}}+\1_{\{Z^{v,2}\leqp t\}}\right)F(t)\right] dF(t)\\
&=\frac 12 \left(F(Z^{v,1})^2+F(Z^{v,2})^2\right)-F(Z^{v,1}\vee Z^{v,2}).
\end{align*}
%As usual, the subscript $+$ means that we are taking the limit coming from the right. 
%Obviously,
%\begin{align*}
%&\E(U)=\int \left(\widetilde  G(t)-F(t)^2\right)dF(t),\\
%&\E(U^2)=\int \left(\widetilde  G(t)-F(t)^2\right)^2dF(t),\\
%&\E(V)=\int \left(F(t)^2-\widetilde  G(t)\right)dF(t),\\
%&\E(V^2)=\frac 12 \int F(t)^4dF(t)+\iint \left[F(t\vee s)\left(F(t\vee s)-F(t)^2-F(s)^2\right)+\frac 12 F(t)^2F(s)^2\right]dG(t,s).
%\end{align*}
By independence, the limiting variance $\xi^2$ is 
\begin{align}\label{eq:var_xi}
\xi^2=\Var U+\Var V.
\end{align}
\end{proof}

\subsection{Proof of Proposition \ref{prop:ex}}

\begin{proof}[Proof of Proposition \ref{prop:ex}]
First of all, the distribution function of $Y$ conditioned on $X_1$ is given by
\begin{align*}
F^{(1)}(t)&=\P(Y\leqp t|X_1)=\Phi\left(\frac{\ln t-X_1}{2}\right).
\end{align*}
Then
\begin{align*}
N_{2,CVM}^{1}&=\int_{\R} \E\left[(F^{(1)}(t)-F_Y(t))^2\right]f_Y(t)dt\\
&=\int_{\R^+} \E\left[\left(\Phi\left(\frac{\ln t-X_1}{2}\right)-\Phi\left(\frac{\ln y}{\sqrt 5}\right)\right)^2\right]\frac{1}{\sqrt{10\pi}t}e^{-(\ln t)^2/10}dt\\
&=\int_{\R} \E\left[\left(\Phi\left(\frac{\sqrt 5 z-X_1}{2}\right)-\Phi\left(z\right)\right)^2\right]e^{-z^2/10}\frac{dz}{\sqrt{2\pi}}\\
&=\E\left[\left(\Phi(X_2)-\Phi\left(\frac{\sqrt 5 X_2-X_1}{2}\right)\right)^2\right]
\end{align*}
where $X_1$ and $X_2$ are independent standard Gaussian random variables.
In the same way,
\begin{align*}
N_{2,CVM}^{2}&=\E\left[(\Phi(X_2)-\Phi\left(\sqrt 5 X_2-2X_1\right))^2\right].
\end{align*}
Thus we are lead to compute the bivariate function:
\[
\varphi(\alpha,\beta):=\E\left[(\Phi(X_2)-\Phi\left(\alpha X_2-\beta X_1\right))^2\right]
\]
for  $(\alpha,\beta)=(\sqrt 5/2,1/2)$ and $(\alpha,\beta)=(\sqrt 5,2)$.
The term $\E\left[\Phi(X_2)^2\right]$ is
\begin{align*}
\E\left[\Phi(X_2)^2\right]&= \int \Phi(z)^2 f(z)dz=\left[\frac 13 \Phi(z)^3\right]_{-\infty}^{+\infty}=\frac 13.
\end{align*}
We introduce three independent random variables $U$, $U'$ and $V$ distributed as standard Gaussian random variables. Then the term $\E\left[\Phi\left(\alpha X_2-\beta X_1\right)^2\right]$ can be rewritten as
\begin{align*}
\E\left[\Phi\left(\alpha X_2-\beta X_1\right)^2\right]&=\E\left[\Phi\left(\sqrt{\alpha^2+\beta^2} V\right)^2\right]=\E\left[\E\left[\1_{U\leqp \sqrt{\alpha^2+\beta^2} V}\vert V\right]^2\right]\\
&=\E\left[\E\left[\1_{U\leqp \sqrt{\alpha^2+\beta^2} V}\vert V\right]\E\left[\1_{U'\leqp \sqrt{\alpha^2+\beta^2} V}\vert V\right]\right]=\E\left[\E\left[\1_{U\leqp \sqrt{\alpha^2+\beta^2} V}\1_{U'\leqp \sqrt{\alpha^2+\beta^2} V}\vert V\right]\right]\\
&=\E\left[\1_{U\leqp \sqrt{\alpha^2+\beta^2} V}\1_{U'\leqp \sqrt{\alpha^2+\beta^2} V}\right]=\P\left(U\leqp \sqrt{\alpha^2+\beta^2} V,\, U'\leqp \sqrt{\alpha^2+\beta^2} V\right)\\
&=: G(\sqrt{\alpha^2+\beta^2}).
\end{align*}
Integrating by parts, we have
\begin{align*}
G'(a)
%&=2 \int_{\R} z\Phi(az)f(az)e^{-z^2/2}\frac{dz}{\sqrt{2\pi}}\\
&=2 \int_{\R} z\Phi(az)e^{-(a^2+1)z^2/2}\frac{dz}{2\pi}\\
&=-\frac{1}{\pi(a^2+1)} \left(\left[\Phi(az)e^{-(a^2+1)z^2/2}\right]_{-\infty}^{+\infty}-a\int_{\R} f(az)e^{-(a^2+1)z^2/2}dz\right)\\
&=\frac{a}{\pi(a^2+1)}\frac{1}{\sqrt{2a^2+1}}.
\end{align*}
Since $G(1)=1/3$, we get
\begin{align*}
G(a)&=\frac 13 + \int_{1}^a  \frac{x}{\pi(x^2+1)}\frac{1}{\sqrt{2x^2+1}}dx=\frac 13 + \frac{1}{\pi}(\arctan \sqrt{1+2a^2}-\arctan \sqrt 3)=\frac{1}{\pi}\arctan \sqrt{1+2a^2}\\
\end{align*}
and
\begin{align*}
\E\left[\Phi\left(\alpha X_2-\beta X_1\right)^2\right]&=\frac 13 + \frac{1}{\pi}(\arctan \sqrt{1+2(\alpha^2+\beta^2)}-\arctan \sqrt 3)=\frac{1}{\pi}\arctan \sqrt{1+2(\alpha^2+\beta^2)}.
\end{align*}
In the same way, the last term $\E\left[\Phi(X_2)\Phi\left(\alpha X_2-\beta X_1\right)\right]$ is given by
\begin{align*}
\E\left[\Phi(X_2)\Phi\left(\alpha X_2-\beta X_1\right)\right]
&=\P\left(U\leqp V,\, \sqrt{\frac{1+\beta^2}{\alpha^2}}U'\leqp V\right)
\end{align*}
where $U$, $U'$ and $V$ are independent standard Gaussian random variables. Remind  that we only need to consider $(\alpha,\beta)=(\sqrt 5/2,1/2)$ and $(\alpha,\beta)=(\sqrt 5,2)$ in which cases $\sqrt{\frac{1+\beta^2}{\alpha^2}}=1$. Thus the last term equals $1/3$ in both cases. It remains to divide by the normalizing factor 1/6 to get the result.
\end{proof}

\begin{rmk}
In the previous proof, we establish that
\begin{align*}
G(a)&=\P\left(U\leqp a V,\, U'\leqp a V\right)
\end{align*}
is equal to $\frac{1}{\pi}\arctan \sqrt{1+2a^2}$ where $U$, $U'$ and $V$ are independent standard Gaussian random variables. Actually, this result is also a straightforward consequence of  Lemma 4.3 in \cite{AW09} with $X=(aV-U)/\sqrt{a^2+1}$ and $Y=(aV-U')/\sqrt{a^2+1}$. Nevertheless, since our proof is different and elegant, we decide not to skip it.
\end{rmk}

\bibliographystyle{plain}
\bibliography{biblio_MultiPick}

\end{document}